\documentclass{amsart}

\usepackage{amsmath}\usepackage{amssymb}
\usepackage{amsthm}
\usepackage{amscd}
\usepackage{amssymb}
\usepackage{amsfonts}
\usepackage[dvips]{graphicx}
\numberwithin{equation}{section}
\usepackage{subfig}

\usepackage{epsfig}

\usepackage{color}

 \usepackage{lineno}

\newtheorem{lemma}{Lemma}
\newtheorem{corollary}{Corollary}

\newtheorem{theorem}{Theorem}

\def\eqn {\begin{equation}}
\def\eeqn {\end{equation}}

\def\real{{\mathbb R}}
\def\R{\real}

\def\ep{\varepsilon}

\def\lb{\lambda}

\def\C{\mathcal C}

\def\J{\mathcal J}

\def\K{\mathcal K}
\def\L{\mathcal L}

\def\S{\mathcal S}

\def\L{\mathcal L}
\def\mcR{\mathcal R}

\newcommand{\Om}{\Omega}

\newcommand{\be}{\begin{equation}}
\newcommand{\ee}{\end{equation}}

\newcommand{\mcs}{{\mathcal S}}

\newcommand{\mcb}{{\mathcal B}}

\newcommand{\mcl}{{\mathcal L}}
\newcommand{\mcr}{\mathcal{R}}

\newcommand{\bdr}{\mathbb{R}}

\newcommand{\mcc}{\mathcal{C}}

\newcommand{\N}{\mathbb N}

\newcommand{\bese}{\begin{subequations}}
\newcommand{\ese}{\end{subequations}}

\begin{document}
\title[Large-amplitude steady downstream water waves]{Large-amplitude steady downstream water waves}
\author[A. Constantin, W. Strauss and E. V\u{a}rv\u{a}ruc\u{a}]{Adrian Constantin, Walter Strauss and
Eugen V\u{a}rv\u{a}ruc\u{a}}

\address{Faculty of Mathematics, University of Vienna, Oskar-Morgenstern-Platz 1, 1090 Vienna, Austria}
\email{adrian.constantin@univie.ac.at}
\address{Brown University, Department of Mathematics and Lefschetz Center for
Dynamical Systems, Box 1917, Providence, RI 02912, USA}
\email{wstrauss@math.brown.edu}%
\address{Faculty of Mathematics, ``Al. I. Cuza" University, Bdul.\ Carol I, nr.\ 11, 700506 Ia\c{s}i, Romania}
\email{eugen.varvaruca@uaic.ro}%

\maketitle

\begin{abstract}
We study wave-current interactions in two-dimensional water flows of constant vorticity over a flat bed.
For large-amplitude periodic traveling waves that propagate at the water surface in the same direction as the underlying current
(downstream waves), we prove explicit uniform bounds for their amplitude.
In particular, our estimates show that the maximum amplitude of the waves becomes vanishingly small as the vorticity increases without limit.
We also prove that the downstream waves on a global bifurcating branch are never overhanging, and that their mass flux and Bernoulli constant are uniformly bounded.
\end{abstract}



\section{Introduction}

Wave-current interactions are ubiquitous since typically a non-trivial mean flow, a current,  underlies surface water waves.
Sheared underlying currents are indicated by the presence of non-trivial vorticity.
The primary sources of currents are winds of long duration \cite{Jon}.
In particular, in shallow regions with nearly flat beds, such as on the continental shelves, systematic studies of the velocity profiles of
wind-generated currents have shown that they are accurately described as flows with constant vorticity \cite{Ew}.

In this paper we consider two-dimensional inviscid steady waves with constant vorticity.
Inviscid theory is the usual framework for
studying water waves that are not close to  breaking because the most significant effects of viscosity in
the open sea produce wave-amplitude reduction, as well as diffusion of the deeper motions, over time scales and
length scales (wave periods and wavelengths) that are far larger than those of the dynamical processes at the surface \cite{DP}.
The choice of constant vorticity is not merely a mathematical simplification. Indeed, when waves propagate at the surface of
water over a nearly flat bed, for waves that are long compared to the mean water depth,
it is the existence of a non-zero mean
vorticity that is more important than its specific distribution (see the discussion in \cite{DP}).
Moreover, in contrast to the substantial research literature on steady three-dimensional irrotational waves,
it turns out that flows of constant vorticity
are inherently two-dimensional (see the discussion in \cite{C, W2}), with the vorticity correlated with the direction of wave-propagation.

The presence of a non-uniform underlying current is experimentally known to drastically alter the behavior of surface waves,
when compared with irrotational waves which
travel from their region of generation through water that is either quiescent or in uniform flow. The nature of a two-dimensional
wave-current interaction notably depends on the directionality of the vertical shear
of the current profile in relation to the direction of wave propagation.
Here we distinguish between {\it favorable} currents, which
are sheared in the same direction as that of the wave propagation ({\it downstream} waves),
and {\it adverse} currents, which are sheared
in the direction opposite to that of wave propagation ({\it upstream} waves).
Field data, laboratory experiments and numerical
simulations (see \cite{CKS1, CKS2, KS1, KS2, Mo, Swan, GPT, tk, V})
lead to the conjecture that an adverse current will shorten
the wavelength, increasing the wave height and the wave steepness, to the extent that bulbous waves appear with
overhanging wave profiles (see \cite{DP, SS}).
On the other hand, for waves propagating downstream, the favorable current
appears to lengthen the wavelength and flatten the wave out, so that its slopes are less steep.

At the present time the state-of-the-art to derive
{\it a priori} bounds on surface water waves of large amplitude lags behind the experimental and numerical developments,
and is to a large extent confined to irrotational deep-water flows.
In the latter setting  \cite{BT, Tol} it was proven that an increase in wave height results in
wave profiles that are not symmetric about their mean level, unlike the small-amplitude
sinusoidal waves familiar from linear theory.
Instead, the crests become higher and the troughs flatter to
the extent that, for a given wavelength, there exists a limiting
wave, the so-called `wave of greatest height' or 'wave of extreme form', which is on the verge of breaking.
This extreme wave is distinguished
by the fact that in the reference frame moving with the wave the water comes to rest at its peaked crest
with included angle $2\pi/3$, as conjectured by Stokes in 1880 and finally proved about a century later in \cite{AFT} (see also \cite{VW1}).
In general, at the wave crest of a traveling surface wave
in irrotational flow with no underlying current, the fluid particles are moving forward at a speed
less than the wave speed (see \cite{C-IM, CSp}).
As the wave profile approaches the wave of greatest height the horizontal particle velocity at the crest
approaches the wave speed until these two velocities become equal in the limiting wave.
A further increase in wave height will cause the fluid particles to overtake the wave itself, and
breaking will ensue (see \cite{PW}).

On the other hand, while it has long been known (see \cite{Mi}) that formal expansions indicate that a uniform vorticity distribution may
accommodate such limiting wave forms and does not alter considerably the shape of their crest (namely, a symmetric peak with
an included angle of $2\pi/3$, as in the case without vorticity), progress towards rigorous results has been much more difficult (see
\cite{V1, VW2}). Apart from their interest in their own right, {\it a priori} bounds on smooth waves of
large amplitude are a necessary prerequisite for the existence of rotational waves of extreme form.

Let us briefly discuss the present state of rigorous mathematical investigations of rotational waves of large amplitude
that are not perturbative, taking advantage instead of structural properties of the equations.
For flows with general vorticity but without stagnation points the existence of surface waves of large amplitude,
with profiles that can be represented as graphs of smooth functions, was established in \cite{CS}.
For flows of constant vorticity an alternative geometric approach \cite{CSV} accommodates stagnation points as well
as the possibility of overhanging profiles.
A number of qualitative studies of rotational waves of large amplitude are also available.
They include symmetry results for wave profiles that are monotone between successive crests and troughs \cite{CEW, CE0},
results about the location of the point of maximal horizontal fluid velocity \cite{CSv, V-SIAM}, and
bounds on the maximum slope of the waves \cite{SW}.
Although such non-perturbative results are not restricted to waves of small amplitude,
their validity has typically required some extraneous information,
such as the absence of stagnation points, the assumption that the wave profile is a graph or the assumption that
some specific bounds (on the velocity or on the amplitude) hold.
Thus the lack of {\it a priori} bounds for flows with vorticity has impeded
substantial progress towards a comprehensive theory for waves of large amplitude.

The main aim of this paper is to derive a uniform bound of the amplitude of downstream waves,
the existence of which has recently been proved in \cite{CSV}.  The approach in \cite{CSV} relies on a
formulation, first introduced there, of the governing equations for steady water waves with constant vorticity
as a one-dimensional nonlinear pseudo-differential equation \eqref{vara}
with a scalar constraint \eqref{meana}. This formulation permits the presence of stagnation points in the flow
as well as overhanging wave profiles.  While from a mathematical point of view it appears at first sight
to be dismayingly complicated, it has a variational structure that warrants a profound analysis, enabling us
to establish, by means of the analytic theory of global bifurcation in a suitable function space, the existence
of two solution curves that contain waves of large amplitude \cite{CSV}, one of this curves consisting of upstream waves,
and the other of downstream waves. While the results in \cite{CSV} left open several possibilities concerning the properties of the waves
corresponding to points on this global curve as the parameter along the curve tends to $\infty$, such as,
for example, those that waves could become overhanging, or that either the parameters in the
problem or suitable norms of the solution could increase without bound, these issues are settled
in the present paper in the case of {\it downstream waves}. A fine analysis of the system (\ref{sys})
that uncovers some unexpected structures leads to the following main result, entirely consistent with
the numerical computations in \cite{KS1,DH}.

\begin{theorem} \label{mainthm}
Consider the waves with constant favorable vorticity $\Upsilon>0$ lying on the bifurcation curve
which is parametrized by $s\in (0,\infty)$.  Then

(i) The waves are not overhanging.

(ii) The wave amplitude (elevation difference between crest and trough) is uniformly bounded along the
bifurcation curve, with an explicit bound depending only on the vorticity and the wave period.

(iii) As a function of the vorticity, the amplitudes of all the waves tend to zero uniformly as the vorticity becomes infinite:
\[\lim\limits_{\Upsilon\to\infty} \sup_{s\in [0,\infty)}\left\{\max\limits_{{\mathcal S}(s)} Y
-\min \limits_{{\mathcal S}(s)} Y \right\}\to 0.\]

(iv) For any $\Upsilon>0$, as one moves along the bifurcation curve, the waves approach their
maximum possible amplitude ${Q(s)}/{2g}$:
$$\quad \lim_{s\to +\infty}  \left\{\max_{{\mathcal S}(s)} Y - \frac {Q(s)}{2g}  \right\} = 0, $$
while $m(s)$ and $Q(s)$ remain uniformly bounded as $s\to\infty$,
where $m(s)$ is the wave flux, $Q(s)$ is the total head and ${\mathcal S}(s)$ is the wave profile.

\end{theorem}

In Section 2 of the present paper we discuss the governing equations and this new formulation
for both upstream and downstream waves.
In particular, the scalar constraint is conductive to an elegant characterization of the {\it downstream} waves.
In Section 3 we prove that the wave profile of each {\it downstream} wave on the solution curve is a graph;
that is, the waves do not overhang.
In Section 4 we derive the key {\it a priori} bounds for the downstream waves.
These bounds use explicit detailed analysis of the Dirichlet-Neumann operator acting on the nonlinear terms.
They are novel, surprising and extremely delicate.
In Section 5 we discuss the physical interpretation of these mathematical results. Some background material
is collected in the Appendix (Section 6).

As already mentioned, Theorem 1 is sufficient to ensure, by a very slight adaptation of the arguments in \cite{V1}
that let $s\to\infty$ along a subsequence,
the existence of a limiting wave with stagnation points at its crests.
An in-depth study of the properties of this limiting wave remains an important open problem.
Another open problem  is the determination of {\it a priori} bounds
and geometric properties of {\it upstream} waves of large amplitude.

\section{Preliminaries}

We consider the problem of two-dimensional spatially periodic travelling free surface gravity water waves in a flow of constant
vorticity over a flat bed. In a frame of reference moving with the speed of the wave, the fluid is in steady flow and occupies a laterally unbounded region $\Omega$ of the $(X,Y)$-plane, whose boundary is made up of two parts: a lower part consisting of the real axis $\mcb=\{Y=0\}$ and representing the flat impermeable water bed, and an upper part that consists of a curve $\mcs$ representing the free surface between the fluid and the atmosphere (see Fig. 1). The steady flow in $\Omega$  may be
described by means of a {\it stream function} $\psi$, so that
the velocity field is $(\psi_Y,-\psi_X)$, and $\psi$ satisfies the following
equations and boundary conditions:
\begin{subequations}\label{g}
\begin{align}
 \Delta\psi&=\Upsilon \quad\text{in } \Omega,\label{g1}\\
 \psi&=0\quad \text{on } {\mathcal S},\label{g3}\\
 \psi&=-m\quad \text{on } {\mathcal B},\label{g4}\\
 \vert\nabla\psi\vert^{2}+2gY&=Q \quad\text{on }{\mathcal S}.\label{g2}
\end{align}
\end{subequations}
Here $g$ is the gravitational constant of acceleration, the constant $m$ is the relative mass flux, the constant $Q$ is the total head, and the vorticity of the flow $\Delta\psi$ is assumed to take the constant value $\Upsilon $. In addition, both the domain $\Omega$ and the stream function $\psi$ will be assumed to be $L$-periodic in the horizontal direction, for some $L>0$. This is a {\it free-boundary problem}, in which both the fluid domain $\Omega$ and the function $\psi$ satisfying (\ref{g}) need to be found as part of the solution.

A {\it stagnation point}
of the flow is a point where $\nabla\psi=(0,0)$. Stagnation points below the surface occur in the case of flow-reversal, even for
waves that are small perturbations of a flat surface (see \cite{CV}). On the other hand, a stagnation
point on the surface $\mcs$ is the hallmark of the limiting `wave of greatest height', which is on the verge of breaking and whose profile may have a corner singularity at the wave crest, see \cite{V1, VW1, VW2}. Such waves present special features that stand apart
from those of regular waves and, in order to avoid technicalities that are of little relevance for the purposes of this paper, will not be investigated directly. Instead, we look for smooth (regular) solutions of \eqref{g} that satisfy
\begin{equation}\label{nss}
Q> 2gY \quad\text{on }{\mathcal S},
\end{equation}
bearing in mind the possibility that singular waves may potentially arise as limits of such regular waves.

In order to obtain existence results for (\ref{g}), it is usually necessary to consider an equivalent reformulation of the problem over some fixed domain.
In the present setting, for waves of large amplitude the most comprehensive results available are due to \cite{CSV}
and are based on the following alternative formulation of the governing equations \eqref{g}:
\begin{subequations}\label{sys}
\be\label{vara}
\begin{aligned}
\mcc_{kh}\Big( (Q-2gv-\Upsilon ^2 v^2)\,v'\Big) &+ \,\displaystyle(Q-2gv-\Upsilon ^2 v^2)\Big(\frac{1}{k} + \mcc_{kh}(v')\Big)\\
 &- \,\displaystyle\,2\Upsilon \,v \,\Big( \frac{m}{kh} - \frac{\Upsilon }{2kh}\,[v^2] - \Upsilon \,\mcc_{kh}(vv')\Big) \\
& - \,\displaystyle\frac{Q-2\Upsilon  m -2gh}{k} + 2g\,[v\,\mcc_{kh}(v')]=0,
\end{aligned}
\ee
\be
\begin{aligned}\label{meana}
\displaystyle\left[\left( \frac{m}{kh} - \frac{\Upsilon }{2kh}\,[v^2]-\Upsilon \mcc_{kh}(vv') +\Upsilon  v\,\Big(\frac{1}{k}+\mcc_{kh}(v')\Big)\right)^2\right] \\
\qquad -\displaystyle\left[(Q-2gv)\,\left((v')^2
+\Big(\frac{1}{k}+\mcc_{kh}(v')\Big)^2\right)\right]=0.
\end{aligned}
\ee
\end{subequations}
The square bracket $[\cdot ]$ will be used throughout the paper to denote, for a periodic function,
its average over a period. Thus \eqref{meana} is merely a scalar equation.
In (\ref{sys}), $v$ is a suitably smooth $2\pi$-periodic function of a real variable $x$ that represents the
wave elevation in a parametrization of the free surface related to a conformal mapping from a horizontal
strip onto the fluid domain (see Fig. 1).
We denote  $h=[v]$, the average of $v$ over a period, which is a positive constant that
may be called the {\it conformal mean depth} of the fluid domain $\Omega$.
The constant $k>0$ is the wave number corresponding to the wave
period $L=2\pi/k$.  The operator $\C_{kh}$ denotes the periodic
Hilbert transform for a strip of height $kh$ (see the Appendix).
As in \eqref{g},  $\Upsilon $, $m$ and $Q$ are real numbers that represent, respectively, the constant vorticity, the relative mass flux, and the total head.

\begin{figure}[h]
  \centering
  \includegraphics[width=0.87\textwidth]{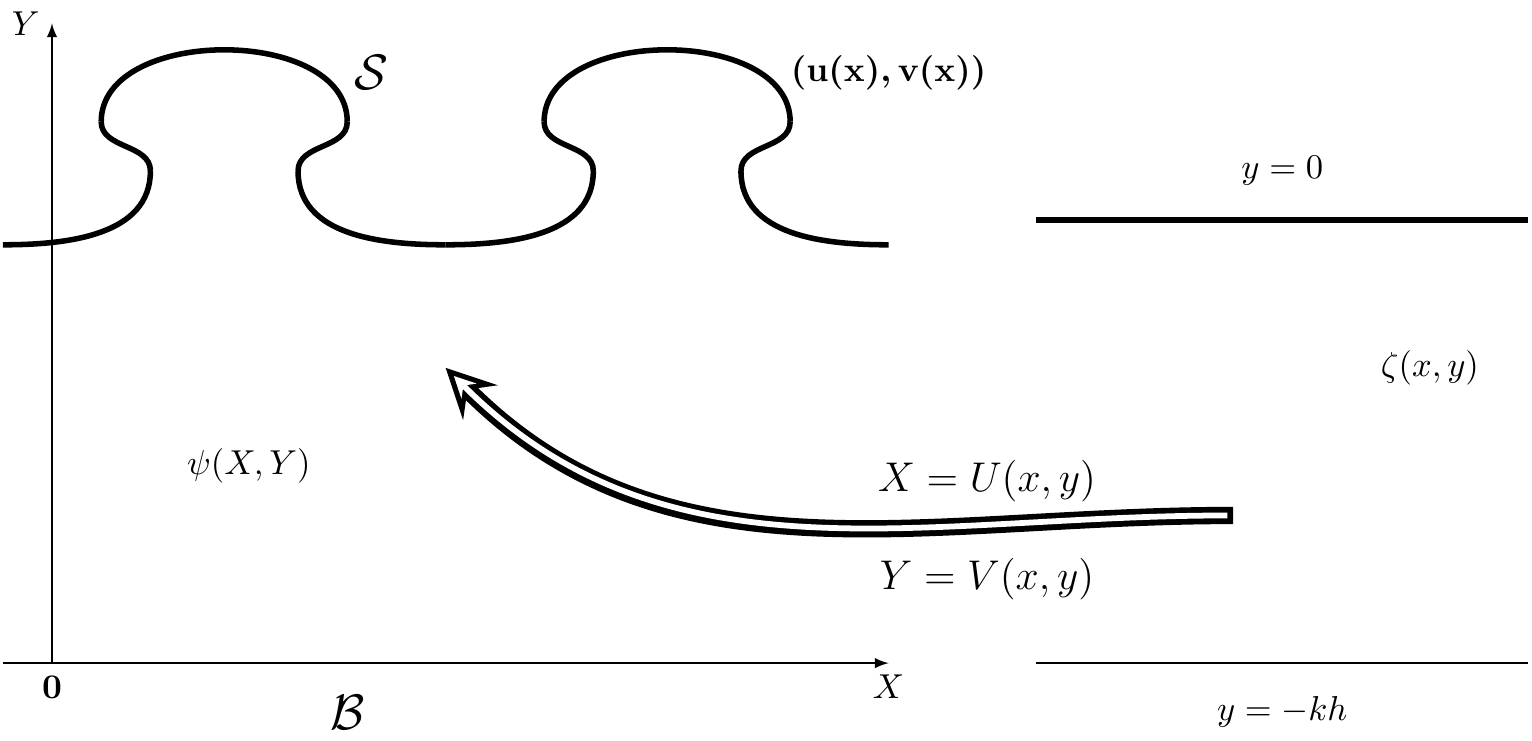}
   \captionsetup{width=.97\linewidth}
  \caption{\footnotesize The conformal parametrisation of the fluid domain: sketch of the horizontal strip ${\mathcal R}_{kh}$ on the right
  and on the left, deptiction of the configuration in a frame moving
  at the wave speed, with the free stationary wave profile parametrised by $(u(x),\,v(x))$ with
  $u(x)=\frac{x}{k}+\big(\mcc_{kh}(v-h)\big)(x)$.}
\end{figure}

Although any (smooth) solution of (\ref{g}) gives rise to a solution of (\ref{sys}), a (smooth) solution of (\ref{sys})
gives rise to a solution of (\ref{g}) if and only if the parametrization $x\mapsto (u(x),\,v(x))$ is regular, that is,
\begin{gather} v(x)>0\qquad\text{for all }x\in\bdr,\label{pos}\\
\text{the mapping $x\mapsto \left(\frac{x}{k}+\big(\mcc_{kh}(v-h)\big)(x),\, v(x)\right)$ is
injective on $\bdr$}\label{m2},\\
\left(v'(x)\right)^2+\left(\frac{1}{k}+\mcc_{kh}(v')(x)\right)^2\neq 0\qquad\text{for all }x\in\bdr\,.    \label{m3}
\end{gather}
If (\ref{pos})-(\ref{m3}) hold for a solution of (\ref{sys}), then a solution $(\Om, \psi)$ of (\ref{g}) can be constructed as described in detail in the Appendix, with the fluid domain $\Om$ being the image of the strip
\[\mcr_{kh}=\{(x,y): x\in \R, -kh<y<0\}\]
through a conformal mapping $U+iV$ obtained easily from $v$. Note also that any solution of (\ref{sys}) also satisfies (see the Appendix)
\begin{equation}\label{add}
\begin{aligned}
\left( \frac{m}{kh} - \frac{\Upsilon }{2kh}[v^2]-\Upsilon \mcc_{kh}(vv') +\Upsilon  v\,\Big(\frac{1}{k}+\mcc_{kh}(v')\Big)\right)^2 \\
\qquad - (Q-2gv)\,\left((v')^2
+\Big(\frac{1}{k}+\mcc_{kh}(v')\Big)^2\right)=0\,.
\end{aligned}
\end{equation}
In view of this, if \eqref{nss} holds, then we see that \eqref{m3} yields
\begin{equation}\label{win}
\frac{m}{kh} - \frac{\Upsilon }{2kh}[v^2]-\Upsilon \mcc_{kh}(vv') +\Upsilon  v\,\Big(\frac{1}{k}+\mcc_{kh}(v')\Big) \neq 0 \quad
\text{on}\quad {\mathbb R}\,.
\end{equation}

The existence of solutions $(m,Q,v)$ of (\ref{sys}) such that $[v]=h$ was studied in \cite{CSV}
in the space $\R\times \R\times C^{2, \alpha}_{2\pi, e}(\bdr)$ ,
where  $$C^{2, \alpha}_{2\pi, e}(\bdr) =\{ f\in C^{2, \alpha}_{2\pi}(\bdr):
\ f(x)=f(-x)\text{ for all }x\in\bdr\},$$
for any fixed H\"{o}lder exponent $\alpha \in (0,1)$. The requirement that $v$ be an even function
reflects the symmetry of the corresponding wave profile about the crest line that corresponds to $x=0$.

Furthermore, seeking wave profiles $\mcs=\{(u(x),v(x)):x\in\R\}$ whose vertical coordinate strictly decreases between
each of its consecutive global maxima and minima, which are unique per principal period, we consider the following properties
of the pair $(m,v)\in\R\times C^{2,\alpha}_{2\pi, e}$:
\begin{gather}
 v(x)>0 \qquad\text{for all } x\in \R,\label{as03}\\
  v'(x)<0\quad\text{for all }x\in \left(0,\pi\right),\label{nmbv}\\\
  v''(0)<0,\,\, v''\left(\pi\right)>0,\label{nodalloc}\\
 0< \frac{x}{k}+ \big(\mcc_{kh}(v-[v])\big)(x)< \frac{\pi}{k}\qquad\text{for all }x\in \left(0,\pi\right),\label{ga1}\\
  \frac{1}{k}+ \big(\mcc_{kh}(v')\big)(0)>0,\quad \frac{1}{k}+ \big(\mcc_{kh}(v')\big)\left(\pi\right)>0,\label{ga2}\\
 \pm \left(\frac{m}{kh} - \frac{\Upsilon }{2kh}[v^2]-\Upsilon \mcc_{kh}(vv') +\Upsilon  v\,\Big(\frac{1}{k}+\mcc_{kh}(v')\Big)\right)> 0
\qquad\text{on } \R\,.\label{aza0}
\end{gather}
 The condition \eqref{aza0} comes from merely \eqref{win}. We define the open sets
\be
\label{v+} {\mathcal V}_{\pm}=\{(m, Q,v)\in \R\times \R\times C^{2, \alpha}_{2\pi, e}(\bdr):
(\ref{as03})-(\ref{aza0}) \text{ hold}\}\,,\ee
where the choice of sign in ${\mathcal V}_{\pm}$ is the same as that in (\ref{aza0}).
We emphasize that the sets ${\mathcal V}_{\pm}$ can
accommodate waves with overhanging profiles since no claim is being made
that the horizontal coordinate of the wave profile is also a strictly monotone function of the parameter $x$.

The constraints \eqref{as03}-\eqref{aza0} ensure, in particular, that the associated free surface $\mcs$ is not flat. Consequently, they exclude the family
of trivial (laminar) solutions of (\ref{sys}) for which $v\equiv h$, and for which $Q$ and $m$ are related by
\be Q=2gh + \Big( \frac{m}{h}+\frac{\Upsilon
h}{2}\Big)^2, \label{zar}\ee
where $m\in\R$ is arbitrary. This family represents a curve
\[{\mathcal K}_{\textnormal{triv}}=\left\{\left(m, 2gh + \Big( \frac{m}{h}+\frac{\Upsilon
h}{2}\Big)^2, h\right): m\in\R\right\}\]
in the space $\R\times \R\times C^{2, \alpha}_{2\pi, e}(\bdr)$.
These trivial solutions
correspond to parallel shear flows in the fluid domain bounded below by the
rigid bed $\mcb$ and above by the free surface $Y=h$, with stream
function
\[\psi(X,Y)=\frac{\Upsilon }{2}Y^2+\left(\frac{m}{h}-\frac{\Upsilon
h}{2}\right)Y-m, \qquad X\in\bdr, 0\leq Y\leq h,\] and velocity field
\begin{equation}\label{lb8}
(\psi_Y,-\psi_X)=\Big(\Upsilon  Y+\frac{m}{h}-\frac{\Upsilon  h}{2}
,0\Big), \qquad X\in\bdr, 0\leq Y\leq h.
\end{equation}

Returning to the general case, it is explained in the Appendix how a solution $(\Omega, \psi)$ of (\ref{g}) can be constructed from
a solution of (\ref{sys}) by means of a conformal map $U+iV$ from $\mcr_{kh}$ onto $\Omega$, an important role being played by
a function $\zeta$ on $\mcr_{kh}$ that satisfies the system \eqref{gc} and is related to $\psi$ by
\be\label{zeta}
\zeta(x,y)=\psi(U(x,y),V(x,y)) + m- \frac{1}{2}\,\Upsilon V^2(x,y)\,,\qquad (x,y) \in \mcr_{kh}.
\ee
It then follows that the fluid velocity at the location $(X,Y)=(U(x,y),\,V(x,y)) \in \Omega$ with $(x,y) \in \mcr_{kh}$, is given by
\be\label{vel}
(\psi_Y,-\psi_X)=\Big(\frac{V_x\zeta_x+V_y\zeta_y}{V_x^2+V_y^2}
+\Upsilon  V,\, \frac{V_x\zeta_y-V_y\zeta_x}{V_x^2+V_y^2}\Big).
\ee

The main existence result for waves of large amplitude, which is based on an application of global real-analytic
bifurcation theory, is as follows (see Theorem 5 in \cite{CSV}).

\begin{theorem}\label{glbp}
Let $h,\,k>0$ and $\Upsilon  \in \bdr$ be given. Set
\be\label{lb7}
m_{\pm}^* = -\displaystyle\frac{\Upsilon  h^2}{2} + \,\displaystyle\frac{\Upsilon  h\tanh(kh)}{2k}
\pm h\,\displaystyle\sqrt{\frac{
\Upsilon ^2\tanh^2(kh)}{4k^2}+g\,\frac{\tanh(kh)}{k}}\,,
\ee
\be \label{qlk7} Q_{\pm}^*= 2gh + \Big( \frac{m_{\pm}^*}{h}+\frac{\Upsilon  h}{2}\Big)^2.\ee
			Then for any real $m \neq m_{\pm}^*$, there exists a neighborhood in
$\R\times \R\times C^{2, \alpha}_{2\pi, e}(\bdr)$ of the point $(m, Q, h)$ on ${\mathcal K}_{\textnormal{triv}}$, where
$Q$ is related to $m$ by {\rm (\ref{zar})}, within which there are no solutions of {\rm (\ref{sys})} in ${\mathcal V}_{\pm}$.
On the other hand, for either choice of sign in $\pm$, there exists in the space
$\R\times \R\times C^{2, \alpha}_{2\pi, e}(\bdr)$ a continuous curve
\be \mathcal{K}_{\pm}=\{(m(s), Q(s), v_s):\ s\in (0,\infty)\}\ee
of solutions of {\rm (\ref{sys})}, such that the following properties {\rm (i)-(vi)} hold:
\newline Local behavior:
\begin{itemize}
\item[(i)]
$\lim_{s \downarrow 0} (m(s), Q(s), v_s)=(m_{\pm}^*, Q_{\pm}^*, h)$;
\item[(ii)] $v_s(x) = h+s\cos(x) + o(s)$ in $C^{2, \alpha}_{2\pi, e}(\bdr)$ as $s\downarrow 0$;
\item[(iii)] there exist a neighbourhood ${\mathcal W}_{\pm}$ of $(m_{\pm}^*, Q_{\pm}^*, h)$ in $\R\times \R\times C^{2, \alpha}_{2\pi, e}(\bdr)$ and $\varepsilon>0$ sufficiently small such that
\[\{(m,Q,v) \in {\mathcal W}_{\pm}:\ v \not\equiv h, \text{{\rm (\ref{sys})} holds}\}= \{(m(s), Q(s), v_s) :\ 0< |s| < \varepsilon \},\]
where, for any $s\in (-\infty,0)$, we have defined
\begin{equation}
 m(s):=m(-s),\quad Q(s):=Q(-s),\quad v_{s}(\cdot):= v_{-s}(\cdot+\pi).\label{ants}
\end{equation}
\end{itemize}
Global behavior:
\begin{itemize}
\item[(iv)]
$Q(s)-2gv_s(x)>0$ for all $x\in \R$ and $s\in (0,\infty)$;
\item[(v)]
$\mathcal{K}_{\pm}$ has a real-analytic reparametrization locally around each of its points;
\item[(vi)] one of the following alternatives occurs:
\begin{itemize}
\item[$(A_1)$] either $\mathcal{K}_{\pm}\subset {\mathcal V}_{\pm}$ and
\be\label{coonc}
\begin{array}{l}
\displaystyle\min\Big\{\frac{1}{1+||(m(s), Q(s), v_s)||_{\R\times\R\times C^{2, \alpha}_{2\pi, e}(\bdr)}},\,\\
\qquad\qquad\qquad\qquad \displaystyle\min_{x\in\R}\big\{Q(s)-2gv_s(x)\big\}\Big\}
\to 0 \qquad\text{as }s\uparrow\infty;
\end{array}\ee
\item[$(A_2)$] or there exists some $s^\ast\in (0,\infty)$ such that:
\[(m(s), Q(s), v_s)\in {\mathcal V}_{\pm} \quad\text{for all }s\in (0,s^\ast), \]
whereas $(m(s^\ast), Q(s^\ast), v_{s^\ast})$ satisfies {\rm (\ref{as03})}--{\rm (\ref{nodalloc})},
{\rm (\ref{ga2})} and {\rm (\ref{aza0})}, while instead of {\rm (\ref{ga1})} it satisfies
\begin{subequations}\label{ga}
\begin{align} 0< \frac{x}{k}+ \big(\mcc_{kh}(v-h)\big)(x)&\leq \frac{\pi}{k}\qquad\text{for all }x\in (0,\pi),\label{ga3}\\
 \frac{x_0}{k}+ \big(\mcc_{kh}(v-h)\big)(x_0)&=\frac{\pi}{k}\qquad\text{for some }x_0\in (0,\pi).\label{ga4}\end{align}
 \end{subequations}
\end{itemize}
\end{itemize}
\end{theorem}

In this result, (\ref{lb7}) and (\ref{qlk7}) identify the local bifurcation points along the trivial solution curve ${\mathcal K}_{\textnormal{triv}}$,
(i)--(iii) describe the local behavior of the curve,
and (iv)--(vi) describe the global behavior.  The alternative ($A_1$) means that either the curve
is unbounded in the function space or it approaches stagnation at the crest (and thus, may have waves of greatest height as limit points).
The alternative ($A_2$) means that solutions on  ${\mathcal K}_{\pm}$ that correspond
to physical water waves with qualitative properties as described by (\ref{as03})-(\ref{aza0}) do exist
{\it until} a limiting configuration
whose profile self-intersects on the line strictly above the trough is reached at $s=s^\ast$.

Note that for the laminar flows given by (\ref{lb8}), the horizontal velocity at the free
surface is $\displaystyle\Big(\frac{m}{h}+\frac{\Upsilon  h}{2}\Big)$. Introducing the parameter
\be\lambda=\frac{m}{h}+\frac{\Upsilon  h}{2},\label{lamb}\ee
it is observed
in \cite{CSV}  that, for  a flow with a flat free surface at which nonlinear small-amplitude waves bifurcate,
the horizontal velocity at the surface is given by
\begin{equation}
\lambda_{\pm}^*=\frac{\Upsilon \tanh(kh)}{2k} \pm
\displaystyle\sqrt{\frac{
\Upsilon ^2\tanh^2(kh)}{4k^2}+g\,\frac{\tanh(kh)}{k}}.\label{lb6}
\end{equation}
Note that $\lb_+^*>0$ and $\lb_-^*<0$, so that there are no stagnation points on the free surface $\mcs$ for the waves of small amplitude whose existence is guaranteed by
local bifurcation. This property holds for all the genuine waves provided by Theorem \ref{glbp}, even for the waves
of large amplitude that are not merely small perturbations of a laminar flow.  It is also noted in \cite{CSV}  that,
for any solution $(m,Q,v)$ of (\ref{sys}), the function $v$ is necessarily smooth (of class $C^\infty$) on $\R$.

Let us also mention that \cite{CSV} contains in fact a more comprehensive global bifurcation theory for (\ref{sys}) than that in Theorem \ref{glbp}, where we have restricted attention only to nontrivial waves with the nodal properties expressed by {\rm $(\ref{as03})-(\ref{aza0})$}.

\section{On the favorable branch the waves do not overhang.}
Numerical simulations (for instance \cite{KS2})  clearly indicate that the behavior of the solutions on the
branches ${\mathcal K}_{\pm}$ can be quite different,
depending on the choice $\pm$ of the branch and on the sign of $\Upsilon$.
The main result of this section partly confirms the numerical observations, ruling out alternative $(A_2)$ on
the favorable bifurcating branch.
Note that the equations \eqref{sys} are invariant under the change of parity $\Upsilon \to -\Upsilon$, $m\to -m$.
This simply reverses the vorticity and the direction of the flow in the moving frame.  The favorable case is
represented either by the curve ${\mathcal K}_-$ with $\Upsilon>0$ or by the curve ${\mathcal K}_+$ with $\Upsilon<0$.
Thus we may consider just the former case.
The following theorem establishes that, all along the branch, the free surface of the wave is the graph of a function and no flow reversal occurs within the corresponding fluid domain.

\begin{theorem}\label{gra}
Let $\Upsilon \geq 0$.
Then, in the notation of Theorem \ref{glbp}, the bifurcating curve ${\mathcal K}_-$ of
solutions  of  {\rm (\ref{sys})} satisfies alternative $(A_1)$.
Moreover, any solution $(m,Q,v)$ on ${\mathcal K}_-$ satisfies
\be \frac{1}{k}+ \big(\mcc_{kh}(v')\big)(x)>0\quad\text{for all }x\in\R\label{graph},\ee
and, if $(\Omega, \psi)$ denotes the corresponding solution of {\rm (\ref{g})},  then
\be \psi_Y <0 \quad\text{in }\overline\Omega.\label{nost}\ee
\end{theorem}

\begin{proof} The proof that follows is similar in spirit to that of Theorem \ref{glbp}, being based on a continuation argument that shows that property (\ref{graph}) is preserved all along the curve ${\mathcal K}_-$. Since (\ref{graph}) represents a strengthening of (\ref{ga1}) and (\ref{ga2}), and since alternative $(A_2)$ involves the failure of (\ref{ga1}), the global validity of (\ref{graph}) prevents the occurrence of $(A_2)$.

Though perhaps not impossible, it appears difficult to study the validity of (\ref{graph}) in isolation, and thus our approach will be to study it in conjunction with all the other properties that occur in Theorem \ref{glbp}. We consider therefore in what follows the set
\be{\mathcal U}_-=\{(m, Q,v)\in {\mathcal V}_{-}: (\ref{graph}) \text{ holds}\},\ee
and define
 \be J=\{s\in (0,\infty): (m(s), Q(s), v_s)\in {\mathcal U}_-\}.\ee
 Since ${\mathcal U}_-$ is an open set in  $\R\times \R\times C^{2, \alpha}_{2\pi, e}(\bdr)$, $J$ is an open subset of $(0,\infty)$.
We claim that  $J=(0,\infty)$.
To that aim, it is necessary to revisit some of the arguments in the proof of Theorem \ref{glbp} in \cite{CSV} .

{\it We argue by contradiction}, assuming that the open set $J$ is not the whole of $(0,\infty)$.
Of course it is immediate that there exists $\ep>0$ such that $(0,\ep)\subset J$.
Let $s^\star$ be the upper endpoint of the largest open interval that contains $(0, \ep)$ and is contained in $J$.
Then $(0,s^\star)\subset J$ and $s^\star\notin J$.
In what follows we shall investigate the properties of the solution $(m(s^\star), Q(s^\star), v_{s^\star})$.

The fact that necessarily $v_{s^\star}\not\equiv h$ follows by the same argument as in \cite{CSV},
which involves the knowledge of all local bifurcation points
on ${\mathcal K}_{\textnormal{triv}}$, the different nodal patterns of the solutions on the local bifurcating curves
(expressed by conditions such as (\ref{ants}), (\ref{aza0}) and (\ref{nmbv})), and the specific manner of
construction of the global curves in the real-analytic theory of Dancer, Buffoni and Toland.

 For notational simplicity, we shall denote $(m(s^\star), Q(s^\star), v_{s^\star})$
 by $(m,Q, v)$. This is a limit of solutions satisfying \eqref{as03}--\eqref{aza0} and (\ref{graph}).
 The definition of $s^\star$ implies that the following
 {\it non-strict} inequalities hold:
\begin{gather} v(x)\geq 0 \quad\text{for all }x\in\R, \label{as5}\\
 v'(x)\leq 0\quad\text{for all }x\in [0,\pi],  \label{as1}\\
 0\leq \frac{x}{k}+ \big(\mcc_{kh}(v-h)\big)(x)\leq \frac{\pi}{k}\quad \text{for all }x\in [0,\pi],  \label{as2}\\
\frac{m}{kh} - \frac{\Upsilon }{2kh}[v^2]-\Upsilon \mcc_{kh}(vv') +\Upsilon  v\,\Big(\frac{1}{k}+\mcc_{kh}(v')\Big)\leq 0
\quad\text{on }\R,\label{aza}\end{gather}
as well as
\be \frac{1}{k}+ \big(\mcc_{kh}(v')\big)(x)\geq 0\quad\text{for all }x\in\R\label{gra0}.\ee
Now an examination
of the arguments in Section 4 of \cite{CSV} shows that the inequalities (\ref{as5})--(\ref{aza}), together with the fact that $(m,Q,v)$
 is the limit of a sequence of solutions that correspond to solutions of (\ref{g}) in the physical plane
 (which is the case here because of the definition of $s^\star$), are enough to ensure that
 the strict inequalities (\ref{as03})--(\ref{nodalloc}) and (\ref{ga2}) are satisfied, while
  (\ref{aza0}) holds in the form
  \be\frac{m}{kh} - \frac{\Upsilon }{2kh}[v^2]-\Upsilon \mcc_{kh}(vv') +\Upsilon  v\,\Big(\frac{1}{k}+\mcc_{kh}(v')\Big)< 0
\quad\text{on }\R.\label{aza2}\ee
 On the other hand,
 if (\ref{graph}) were also satisfied, then it would imply the validity of
 (\ref{ga1}), and thus would contradict the definition of $s^\star$.
 Therefore (\ref{graph}) necessarily fails. This property, combined with the validity of (\ref{ga2}) and the periodicity and evenness of $v$, ensure that there exists $x_0\in (0,\pi)$ such that
\be \frac{1}{k}+ \big(\mcc_{kh}(v')\big)(x_0)= 0.\label{vert}\ee
Note also that the validity of (\ref{ga2}) and (\ref{gra0}) implies that (\ref{ga1}) also holds. As explained in \cite{CSV}, in combination with the evenness and monotonicity of $v$, this ensures the validity of (\ref{m2}). Moreover, since (\ref{pos}) also holds, and (\ref{m3}) holds as a consequence of (\ref{aza2}) and (\ref{add}), it follows that the solution $(m,Q,v)$ under discussion, which is in fact $(m(s^\star), Q(s^\star), v_{s^\star})$, corresponds to a solution $(\Omega,\psi)$ of (\ref{g}). We now work in the physical plane and show that such a solution, with all the properties that have been established so far, cannot exist.

For simplicity, we denote
\be u(x)=\frac{x}{k}+ \big(\mcc_{kh}(v-h)\big)(x)\quad\text{for all }x\in \R,\label{uuu}\ee
so that $u(x)=U(x,0)$ for all $x\in \R$, where $U$ is a harmonic conjugate of $-V$, the holomorphic function $U+iV$ being a conformal mapping from $\mcr_{kh}$ onto $\Omega$.
Then $\mcs$, the top boundary of $\Omega$, admits the parametrization
\be \mcs=\{(u(x),v(x)):x\in\R\},\label{eqss}\ee
whose properties may be summarized for easy reference as follows:
\begin{gather}
\text{$v$ is $2\pi$-periodic and even,}\label{vet0}\\
\text{$v'(x)<0$ for all $x\in (0,\pi)$},\\
\text{$u$ is odd, $x\mapsto\left(u(x)-\frac{2\pi}{k} x\right)$ is $2\pi$-periodic,}\label{vet3}\\
\text{$(u'(x))^2+(v'(x))^2>0$ for all $x\in\R$},\label{vet4}\\
\text{$u'(0)>0$, $u'(\pi)>0$},\label{vet1}\\
\text{$u'(x)\geq 0$ for all $x\in \R$},\label{vet2}\\
\text{there exists $x_0\in (0,\pi)$ such that $u'(x_0)=0$}.\label{veta}
\end{gather}

It is \eqref{veta} that will lead to the contradiction.
Let us examine the sign of $\psi_Y$ in $\overline\Omega$.
Note first that, as a consequence of (\ref{gc3}), we have on $y=0$ that
\[\zeta_x=-\Upsilon  VV_x,\]
which, when substituted in the formula (\ref{vel}) for the velocity field, and
using also the Cauchy--Riemann equations, leads to the following relations, for all $x\in\R$:
\begin{align}
\psi_Y(u(x),v(x))
&= \frac{u'(x)\{\zeta_y(x,0)+\Upsilon  V(x,0) V_y(x,0)\}}{(u'(x))^2+(v'(x))^2},\label{fpi}
\\ -\psi_X(u(x),v(x))
& =\frac{v'(x)\{\zeta_y(x,0)+\Upsilon  V(x,0) V_y(x,0)\}}{(u'(x))^2+(v'(x))^2}.\label{fpx}
\end{align}
Observe that, by (\ref{vzc}), we have on $y=0$ that
\[\zeta_y+\Upsilon  V V_y= \frac{m}{kh} - \frac{\Upsilon }{2kh}[v^2]-\Upsilon \mcc_{kh}(vv') +\Upsilon  v\,\Big(\frac{1}{k}+\mcc_{kh}(v')\Big),\]
and this quantity is strictly negative, as (\ref{aza0}) shows.
It follows from (\ref{fpi}) and \eqref{vet2} that
\be\psi_Y\leq 0\text{ on }\mcs,\quad \psi_Y(u(x_0),v(x_0))=0\quad\text{and}\quad\psi_Y\not\equiv 0 \text{ on }\mcs.\label{sign}\ee
But as a consequence of (\ref{g1}), the function $\psi_Y$ is harmonic in $\Omega$ and satisfies
\[(\psi_Y)_Y=\Upsilon \geq 0 \quad\text{on the bottom $\mcb$ of the fluid domain}.\]
 Thus, by the strong maximum principle and the Hopf boundary point lemma, the maximum of $\psi_Y$ over the periodic domain $\overline\Omega$ can be attained neither in $\Omega$ nor on $\mcb$. It is now a consequence of (\ref{sign}) that
 \be \psi_Y<0\quad\text{in }\Omega\cup \mcb.\label{pl}\ee
Also, it is a consequence of (\ref{fpi}), (\ref{fpx}) and (\ref{aza2}) that $|\nabla\psi|\neq 0$ on $\mcs$. In combination with (\ref{pl}), this implies that
 \[|\nabla\psi|\neq 0\quad\text{in }\overline\Omega.\]

In order to rule out the existence of a solution of (\ref{g}) for which the free boundary $\mcs$ has the
properties (\ref{vet0})--(\ref{veta}), we follow an idea that goes back to Spielvogel \cite{Sp},
and consider in $\overline{\Omega}$ the function $R$ given by \be
R:=\frac{1}{2}\vert\nabla\psi\vert^{2}+gY-\frac{1}{2}Q-\Upsilon \psi.\label{r}\ee
The function $R$ is in fact, up to an additive constant, the negative of the fluid pressure.
A direct calculation (see \cite{SW, V1}) shows that $R$ satisfies in $\Omega$
\be \Delta R-\frac{2(R_X+\Upsilon \psi_X)}{|\nabla
\psi|^2}R_X-\frac{2(R_Y-2g+\Upsilon \psi_Y)}{|\nabla \psi|^2}R_Y=\frac{2g}{|\nabla\psi|^2}(g-\Upsilon \psi_Y).\label{elef}
\ee
(Note that the definition of vorticity $\gamma$
considered in \cite{SW, V1} differs by sign from the definition $\Upsilon$ we are considering here.)
One can also easily check that \be R_Y=g>0\quad\text{on
}\mcb\label{bco1}.\ee
As a consequence of (\ref{pl}) and the assumption $\Upsilon \geq 0$, the right-hand side in (\ref{elef}) is positive.
It follows from (\ref{elef}) and (\ref{bco1}) that the maximum over $\overline\Omega$ of $R$
can be attained neither in $\Omega$ nor on $\mcb$,
and therefore is attained all along $\mcs$, since $R=0$ there by (\ref{g3}) and (\ref{g2}).
From the Hopf boundary point lemma we infer that the normal derivative of $R$ has a strict sign
all along $\mcs$, and in particular at the point $(u(x_0),v(x_0))$, at which the tangent is vertical by \eqref{veta} and \eqref{sign}.
We therefore deduce that at $(u(x_0), v(x_0))$ we have
 \be 0\neq R_X=\psi_X\psi_{XX}+\psi_Y\psi_{XY}-\Upsilon  \psi_X=-\psi_X\psi_{YY},\label{mbv}\ee
 where we have taken into account (\ref{g1}) and the fact  from \eqref{sign} that $\psi_Y=0$ at that point.

  On the other hand, twice differentiating
  \[\psi(u(x),v(x))=0\quad\text{for all }x\in \R\]
   with respect to $x$, evaluating the result at $x_0$, and using $\psi_Y(u(x_0),v(x_0))=0$ and $u'(x_0)=0$,
   we find that
  \be \psi_X(u(x_0),v(x_0))u''(x_0)+\psi_{YY}(u(x_0),v(x_0))(v'(x_0))^2=0.\label{derd}\ee
  Note, however, that (\ref{vet2}) and (\ref{veta}) imply that $u''(x_0)=0$.
  Thus taking also into account that $v'(x_0)\neq 0$, it follows from (\ref{derd}) that $\psi_{YY}(u(x_0),v(x_0))=0$.
 This contradicts (\ref{mbv}).

  The source of the contradiction can only be the assumption that $J\neq (0,\infty)$, which must therefore be false. We have thus proved that $J=(0,\infty)$, which means that
  \[{\mathcal K}_-\subset {\mathcal U}_-,\]
  and therefore all solutions on ${\mathcal K}_-$ satisfy (\ref{graph}), as required.

  It remains to prove the validity of (\ref{nost}) for all solutions on ${\mathcal K}_-$. Let $(m,Q,v)$ be an arbitrary such solution, and let $(\Omega,\psi)$ be the associated solution of $(\ref{g})$. Then the top boundary $\mcs$ of $\Omega$ admits the parametrization (\ref{eqss}), where the function $u$ is given by (\ref{uuu}), and now the properties (\ref{vet0})--(\ref{vet3}) hold, while instead of (\ref{vet4})--(\ref{veta}) we now have the condition
  \be u'(x)>0\quad\text{for all }x\in\R.\ee
   Application of the maximum principle
   for $\psi_Y$ in $\Omega$ in a way similar to the preceding part of the proof now leads to (\ref{nost}) instead of (\ref{pl}).
   This completes the proof of Theorem \ref{gra}.
\end{proof}

The validity of \eqref{nost} and the fact that the free surface profile is a graph place the
solutions on ${\mathcal K}_-$ within the framework of the considerations made in \cite{CE}, which yield
the folllowing regularity result.

\begin{corollary}
Let $\Upsilon  \ge 0$. Then any H\"older continuosly differentiable solution $v$ on
the curve ${\mathcal K}_-$ is real analytic.
\end{corollary}

\section{Bound on the amplitudes of the waves}
In this section we derive an explicit bound on the wave amplitudes of the favorable waves. We then give the proof of Theorem \ref{mainthm}.
For brevity we write $v(x)=v_s(x)$, for any $s>0$ and all $x\in\R$.

\begin{theorem}   \label{bound} Let $\Upsilon\geq 0$. Then,
along the whole global bifurcation curve ${\mathcal K}_-$, we have the estimate
\begin{subequations}\label{Fii}
\begin{equation}
v(0)-v(\pi) \le \sqrt{\frac{36g^2}{\Upsilon^4}+\frac{24\pi g}{\Upsilon^2k\beta_{hk}(\tfrac{\pi}{2})}} -\frac{6g}{\Upsilon^2}  < \frac{2\pi}{k\beta_{kh}(\tfrac{\pi}{2})} \le
\frac{2\pi^2}{(\pi-2)k}\quad \text{if}\quad \Upsilon>0,
\end{equation}
and
\be v(0)-v(\pi) < \frac{2\pi}{k\beta_{kh}(\tfrac{\pi}{2})} \le
\frac{2\pi^2}{(\pi-2)k}\quad\text{if}\quad \Upsilon=0. \ee
\end{subequations}
\end{theorem}

In order to prove this theorem,  it is convenient to associate,
 to any solution on the global bifurcation curve ${\mathcal K}_-$, the function
\begin{equation}\label{f}
f(x)=\frac{k}{2g}\,\Big(Q-2g\,v(x)\Big)\,\quad\text{for all }x\in\R,
\end{equation}
which  is smooth, even and $2\pi$-periodic on $\R$.
Note that $[v]=h$ yields
\begin{gather}
 [f]=\frac{k(Q-2gh)}{2g} \,,\label{f5}\\
[v^2]=\frac{Qh}{g}- \frac{Q^2}{4g^2} + \frac{[f^2]}{k^2}\,.\label{f6}
\end{gather}
Due to Theorem \ref{glbp}(iv), \eqref{nmbv}, \eqref{graph} and \eqref{aza0}, respectively, we have
 \begin{gather}
 f(x)>0 \quad\text{for all } x\in \R,\label{f1}\\
 f'(x)>0\quad\text{for all } x\in \left(0,\pi\right),\label{f2}\\
 (\mcc_{kh}(f'))(x) <1\quad\text{for all } x\in \R,\label{f3}\\
 \frac{m}{h} + \frac{\Upsilon  Q^2}{8hg^2}- \frac{\Upsilon }{2k^2h}\,[f^2]
 - \frac{\Upsilon }{k}\,f + \frac{\Upsilon }{k} \Big\{ f \mcc_{kh}(f') - \mcc_{kh}(ff')\Big\} < 0 \quad\text{on } \R. \label{f4}
\end{gather}
Writing
\eqref{vara} in terms of $f$, after multiplication by $\frac{k^2}{2g}$ we obtain the equation
\begin{eqnarray}
\qquad f + (aA + B)f - \tfrac{A}{2}\,f^2 &=&\{ f\mcc_{kh}f' + \mcc_{kh}(ff')\} + b    \label{ef} \\
 &&- \tfrac{A}{2}\,\{ f^2\mcc_{kh}f' + \mcc_{kh}(f^2f') -2f\,\mcc_{kh}(ff')\} \,   \nonumber
\end{eqnarray}
on $\R$, where $A$, $B$, $a$ and $b$ denote the constants
\begin{eqnarray}
\qquad A &=& \frac{\Upsilon ^2}{kg} \ge 0 \,,\hskip 1cm
a=\frac{k}{2g}\,(Q-2gh)\,,\label{A}\\
\quad B &=& \frac{\Upsilon }{g} \Big\{ \frac{m}{h} +\Upsilon  h- \frac{\Upsilon [v^2]}{2h}\Big\}=
\frac{\Upsilon }{g} \Big\{ \frac{m}{h} +\Upsilon  h- \frac{\Upsilon  Q}{2g}
+ \frac{\Upsilon  Q^2}{8hg^2} - \frac{\Upsilon [f^2]}{2k^2h}\Big\}\,\label{B}\,,\\
\quad b &=& \frac{kQ}{2g} - \frac{\Upsilon  k m}{g} - kh - [f \mcc_{kh}(f')]
+ \frac{\Upsilon  k^2 Q}{2g^2}\,B  + \frac{\Upsilon ^2 k Q^2}{8g^3} \label{beta}\,.
\end{eqnarray}
Note that, multiplying  \eqref{f4} by ${\Upsilon }/{g}$, and using the fundamental assumption that $\Upsilon  \ge 0$,  we obtain
\begin{equation}\label{f44}
aA +B - Af + A \{ f \mcc_{kh}(f') - \mcc_{kh}(ff')\} \le 0 \quad\text{on}\quad \R.
\end{equation}

Our approach relies on some structure-exploiting integral representations of the
cubic and quadratic terms in \eqref{ef}, as an effective
tool to obtain estimates. Let us recall from Appendix A in \cite{CSV} that, for
any smooth $2\pi$-periodic function $F: {\mathbb R} \to {\mathbb R}$ with mean zero over each period,
we have
\begin{equation}\label{fc0}
(\mcc_{kh}(F))(x)=\frac{1}{2\pi} \text{PV} \int_{-\pi}^\pi \beta_{kh}(x-s)F(s)\,ds\,,\qquad x \in \R\,,
\end{equation}
where (with $d=kh$) the kernel $\beta_{d}: \R \setminus 2\pi {\mathbb Z}\to {\mathbb R}$,
is given by
\begin{align}  \label{betadef}
\beta_{d}(s) &= -\frac{s}{d}  +  \frac{\pi}{d}\coth\Big(\frac{\pi s}{2d}\Big)  +  \frac{\pi}{d} \sum_{n=1}^{\infty}
\frac{2\sinh(\frac{\pi s}{d})}  {\cosh(\frac{\pi s}{d})-\cosh(\frac{2\pi^2n }{d})}\\
&=  -\frac{s}{d} + \frac\pi{d}  \sum_{n=-\infty}^\infty  \left\{ \coth\left(\frac\pi{2d}(s-2\pi n)\right) + \text{sgn}(n) \right\}.\nonumber
\end{align}
It is $2\pi$-periodic, odd, and smooth on $\bdr \setminus 2\pi{\mathbb Z}$.
					The function
$\beta_{kh}(s)-\frac{\pi}{kh}\coth\Big(\frac{\pi s}{2kh}\Big)$   is continuous at $s=0$.
Although an explicit representation of the kernel $\beta_{kh}$ in terms of three
Jacobi elliptic functions is provided in \cite{AFSS}, our series representation has the advantage that its term-by-term
differentiation reveals that $\beta_{kh}$ is strictly decreasing from $+\infty$ to $-\infty$ on $(0,2\pi)$.
This is an important property that is not obvious from the explicit closed-form representation.
In our proof we will need the following property of $\beta_d$.

\begin{lemma}  For all $d>0$,
$\beta_d(s)$ is a positive function of $s\in (0,\pi]$ strictly decreasing from $+\infty$ to $0$.  Furthermore,
\be \label{strpos}
 \beta_d \left(\pi/2\right)\geq \frac{\pi-2}{\pi}\quad\text{for all }d\in (0,\infty).    \ee
\end{lemma}

\begin{proof}
From \eqref{betadef} it is clear that the series converges uniformly in $[\delta, 2\pi-\delta]$
for all $\delta\in (0,\pi]$, so that
$s\mapsto \beta_d(s)$ is continuous there.  Furthermore, term-by-term differentiation shows that
$s\mapsto \beta_d(s)$ is strictly decreasing on $(0,2\pi)$.
The oddness and $2\pi$-periodicity imply that $\beta_d(\pi)=0$, while the fact that $\lim_{s\downarrow 0} \beta_d(s)=\infty$ is immediate.

 For convenience in proving \eqref{strpos}, let us write $x=\pi^2/2d$ and consider $\beta_d(\pi/2)$
 as a function of $x\in(0,\infty)$.  Then
$$
\pi\beta_d(\pi/2)  =  2x\coth(x/2) -x  -  4x\sinh(x) \sum_{n=1}^\infty \frac1{\cosh(4n x)-\cosh(x)}.  $$

We claim that, for any $x\in (0,\infty)$ and any $\theta\in (1,\infty)$, we have
\[0<\frac{x\sinh(x)}{\cosh(\theta x)-\cosh(x)}\leq \frac{2}{\theta^2-1}.\]
To that aim, we examine the power series expansions of the numerator and denominator in the left side. Note that
\[x\sinh(x)=\sum_{k=1}^{\infty}\frac{1}{(2k-1)!}x^{2k}\quad\text{for all }x\in \R,\]
\[\cosh(\theta x)-\cosh(x)=\sum_{k=1}^\infty \frac{1}{(2k)!}(\theta^{2k}-1)x^{2k}\quad\text{for all }x\in \R.\]
The required result is obtained by simply comparing the coefficients of each even power of $x\in\R$, using the Bernoulli inequality:
\[\theta^{2k}-1= (1+(\theta^2-1))^k-1\geq k(\theta^2-1)\]
for each $\theta\in (1,\infty)$ and $k\in\N$.

It therefore follows that
\[\pi\beta_d(\pi/2)  \geq   2x\coth(x/2) -x  -8\sum_{n=1}^{\infty}\frac{1}{16n^2-1}.\]

We now use the fact that
\[4-8\sum_{n=1}^{\infty}\frac{1}{16n^2-1}=\pi,\]
which is obtained by taking $s=\pi/2$ in the well-known identity (see Section 5.2.1 in \cite{Ahl})
\[\frac{2}{s}+\sum_{n=1}^\infty \frac{4s}{s^2- (2\pi n)^2}=\cot\left(\frac{s}{2}\right)\quad\text{for all }s\in (0,2\pi),\]
to deduce that
\[\pi\beta_d(\pi/2)  \geq   2x\coth(x/2) -x  -(4-\pi).\]
Taking also into account the obvious inequalities
\[ \coth y\geq 1 \text{ and } y\coth y\geq 1\quad\text{for all }y\in (0,\infty),\]
with $y=x/2$, we obtain
\begin{align*}\pi\beta_d(\pi/2)  &\geq   x\coth(x/2) + 2\frac{x}{2}\coth(x/2)-x  -(4-\pi)\\
&\geq x+2 -x+\pi-4=\pi-2,
\end{align*}
as required.
\end{proof}

{\bf Remark} {\it Since letting $d\to \infty$ is equivalent to letting $x\to 0$ in the formula
$$
\pi\beta_d(\pi/2)  =  2x\coth(x/2) -x  -  4x\sinh(x) \sum_{n=1}^\infty \frac1{\cosh(4n x)-\cosh(x)},  $$
we get
\[\pi\lim_{d\to \infty}\beta_d(\pi/2)=4-8\sum_{k=1}^{\infty}\frac{1}{16n^2-1}=\pi\,.\]
It is an interesting conjecture whether actually
\[\beta_d(\pi/2)\geq 1\quad\text{for all }d\in (0,\infty)\,.\]
Numerical computation in Octave/Matlab suggests that the issue is quite subtle.}$\hfill\Box$\bigskip

A very useful consequence of \eqref{fc0} is that
\begin{equation}\label{fc}
(\mcc_{kh}(F'))(x)=-\frac{1}{2\pi} \text{PV} \int_{-\pi}^\pi \beta_{kh}'(s)\Big(F(x)-F(x-s)\Big)\,ds\,,\qquad x \in \R\,,
\end{equation}
for any smooth $2\pi$-periodic function $F: {\mathbb R} \to {\mathbb R}$. From \eqref{fc} we infer that
\begin{eqnarray}
&& (\mcc_{kh}(ff'))(x) = \text{PV} \int_{-\pi}^\pi \frac{-\beta_{kh}'(s)}{4\pi} \,\{ f^2(x)-f^2(x-s)\}\,ds\,, \label{fc2} \\
&& \J f(x) := \Big( f\,\mcc_{kh}f' - \mcc_{kh}(ff')\Big)(x) = \int_{-\pi}^\pi \frac{-\beta_{kh}'(s)}{4\pi} \,\{f(x)-f(x-s)\}^2\,ds\,,\label{fc21} \\
&& \K f(x) := \Big(f^2\,\mcc_{kh}f' + \mcc_{kh}(f^2f')- 2f\,\mcc_{kh}(ff')\Big)(x)  \label{fc3} \\
&&\hskip 2.2cm = \int_{-\pi}^\pi \frac{-\beta_{kh}'(s)}{6\pi} \,\{f(x)-f(x-s)\}^3\,ds\,  \nonumber
\end{eqnarray}
for all $x \in \R$.

\begin{proof}[Proof of Theorem \ref{bound}] Our basic equation \eqref{ef}  implies, by subtraction, that
\begin{align}
\Big( f + (aA + B)f &- \frac{A}{2}\,f^2 \Big) \bigg|_0^\pi  \label{ef1}\\
&= (f\mcc_{kh}f') \bigg|_0^\pi + \mcc_{kh}(ff')   \bigg|_0^\pi   - \frac{A}{2}\,  \K f    \bigg|_0^\pi   \,. \nonumber
\end{align}
Since the left side of the above relation can be written as
\[(f(\pi)-f(0))\left(1+aA+B-\frac{A}{2}(f(\pi)+f(0))\right),\]
in order to take advantage of (\ref{f44}) we add to both sides of (\ref{ef1}) the term $\frac{A}{2}\,(f(\pi)-f(0))(\J f(\pi) + \J f(0))$, obtaining:
\begin{align}\label{ef111}
 &(f(\pi)-f(0)) \, \Big( 1+ aA + B - \frac{A}{2}(f(\pi)+f(0)) +\frac{A}{2}\, (\J f(\pi) + \J f(0))  \Big) \\
&= (f\mcc_{kh}f') \bigg|_0^\pi + \mcc_{kh}(ff')   \bigg|_0^\pi   - \frac{A}{2}\,  \K f    \bigg|_0^\pi  + \frac{A}{2}\,(f(\pi)-f(0))(\J f(\pi) + \J f(0)).
\nonumber
\end{align}
Let us denote by ${\mathcal L}f$ and ${\mathcal R}f$, respectively, the left side and the right side of the above relation:
\[\L f:=(f(\pi)-f(0)) \, \Big( 1+ aA + B - A \,\frac{f(\pi)+f(0)}{2} +\frac{A}{2}\, (\J f(\pi) + \J f(0))  \Big),\]
\[\mcR f:=(f\mcc_{kh}f') \bigg|_0^\pi + \mcc_{kh}(ff')   \bigg|_0^\pi   - \frac{A}{2}\,  \K f    \bigg|_0^\pi  + \frac{A}{2}\,(f(\pi)-f(0))(\J f(\pi) + \J f(0)).\]
In what follows, we shall estimate $\L f$ from above and $\mcR f$ from below.

Observe first that, by evaluating inequality (\ref{f44}) at $x=\pi$ and at $x=0$ and taking the average, we obtain
\begin{equation} \label{f45}
aA + B - \frac{A}{2} (f(\pi)+f(0)) + \frac{A}{2}(\J f(\pi) + \J f(0)) \le 0.
\end{equation}
This is the key place where the assumption $\Upsilon\ge0$ is used.  It implies that
\be {\mathcal L}f\leq f(\pi)-f(0).\label{lebo}\ee
The burden of the proof of Theorem \ref{bound} is carried by the following lemma,
which provides a lower bound for $\mcR f$ under very general assumptions on the function $f$.
 In particular, it is no longer assumed that $f$ satisfies any equation.

\begin{lemma}\label{reb} For any smooth $2\pi$-periodic even function $f$ satisfying {\rm (\ref{f1})} and {\rm (\ref{f2})}, we have
\begin{align}    \mcR f \ge \{f(\pi)-f(0)\}\Bigg( \frac{\beta_{kh}(\frac{\pi}{2})}{\pi}\, f(0) &+ \frac{\beta_{kh}(\frac{\pi}{2})}{2\pi}\,\{f(\pi) + f(0)\}\nonumber\\
 & + A\,\frac{\beta_{kh}(\frac{\pi}{2})}{24\pi}\,\{f(\pi)-f(0)\}^2\Bigg).\label{fi}
  \end{align}
\end{lemma}

Taking for granted for the moment the validity of Lemma \ref{reb} (the proof of which will be given later), we proceed with the proof of Theorem \ref{bound}. Thus, by combining (\ref{lebo}) and (\ref{fi}), we obtain
\begin{equation}\label{Fi}
 \frac{24\pi}{\beta_{kh}(\frac{\pi}{2})} \ge 24 f(0) + 12\{f(\pi)+ f(0)\} + A\,\{f(\pi)-f(0)\}^2 \,,
\end{equation}
and therefore
\begin{equation}\label{Fi2}
 A\,\{f(\pi)-f(0)\}^2 + 12\{f(\pi)- f(0)\} -  \frac{24\pi}{\beta_{kh}(\frac{\pi}{2})} <0\,,
\end{equation}
since $f(\pi)>f(0)>0$ by \eqref{f1}-\eqref{f2}. Recalling \eqref{A}, from \eqref{Fi2} we get
\begin{equation}\label{Fiiv}
f(\pi)-f(0) \le
 \sqrt{\frac{36g^2k^2}{\Upsilon^4}+\frac{24\pi gk}{\Upsilon^2\beta_{hk}(\tfrac{\pi}{2})}} -\frac{6gk}{\Upsilon^2}  < \frac{2\pi}{\beta_{kh}(\tfrac{\pi}{2})}\quad\text{if}\Upsilon>0,
\end{equation} and
\be f(\pi)-f(0) \le \frac{2\pi}{\beta_{kh}(\tfrac{\pi}{2})}\quad\text{if}\Upsilon=0.\ee
 The required result now follows taking into account that
\[v(0)-v(\pi) = \frac1k(f(\pi)-f(0)),\]
as well as the estimate in Lemma 1. This completes the proof of Theorem \ref{bound}.
\end{proof}

It now remains to give the proof of Lemma \ref{reb}.

\begin{proof}[Proof of Lemma \ref{reb}]
It is a consequence of \eqref{fc} that
\begin{equation}\label{fc22}
 (f\mcc_{kh}f') \bigg|_0^\pi = \int_{-\pi}^\pi \frac{-\beta_{kh}'(s)}{2\pi} \,\{ f^2(\pi)-f(\pi)f(\pi-s)-f^2(0)+f(0)f(-s) \}\,ds   \,.
\end{equation}
Also, from \eqref{fc21} it follows that
\begin{equation} \label{ef2}
 \J f(\pi) + \J f(0) =
 \int_{-\pi}^\pi \frac{-\beta_{kh}'(s)}{4\pi} \,\Big(\{f(\pi)-f(\pi-s)\}^2 + \{f(0)-f(-s)\}^2 \Big)\,ds\,. 
\end{equation}
Using also \eqref{fc2}, \eqref{fc3} and \eqref{fc22}, we obtain the explicit representation
\begin{align*}
 \mcR f&= \int_{-\pi}^\pi \frac{-\beta_{kh}'(s)}{2\pi} \,\{ f^2(\pi)-f(\pi)f(\pi-s)-f^2(0)+f(0)f(-s) \}\,ds  \nonumber\\
    &+  \int_{-\pi}^\pi \frac{-\beta_{kh}'(s)}{4\pi} \,\{f^2(\pi)-f^2(\pi-s) +f^2(-s)- f^2(0)\}\,ds \nonumber\\
& + \frac{A}{2} \int_{-\pi}^\pi \frac{-\beta_{kh}'(s)}{4\pi} \,\{f(\pi)-f(\pi-s)\}^2 \Big(\{f(\pi)-f(0)\}  -
\frac{2\{f(\pi)-f(\pi-s)\}}{3} \Big)\,ds \nonumber \\
&  + \frac{A}{2} \int_{-\pi}^\pi \frac{-\beta_{kh}'(s)}{4\pi} \,\{f(0)-f(-s)\}^2 \Big(\{f(\pi)-f(0)\}  +
\frac{2\{f(0)-f(-s)\}}{3} \Big)\,ds. \nonumber\\
\end{align*}
Next, we use the facts that $-\beta_{kh}'(s) \ge0$ and that $f(x)$ is maximized at $\pi$ and minimized at $0$, to obtain that
\be \mcR f \ge I + II + III,\label{123}\ee
where
\begin{align*}
I  &=   \int_{-\pi}^\pi \frac{-\beta_{kh}'(s)}{2\pi} \,\{ f^2(\pi)-f(\pi)f(\pi-s)-f^2(0)+f(0)f(-s) \}\,ds,\\
II &=      \int_{-\pi}^\pi \frac{-\beta_{kh}'(s)}{4\pi} \,\{f^2(\pi)-f^2(\pi-s) +f^2(-s)- f^2(0)\}\,ds,\\
III &=     \frac{A}{2}\{f(\pi)-f(0)\} \int_{-\pi}^\pi \frac{-\beta_{kh}'(s)}{12\pi} \,\Big(\{f(\pi)-f(\pi-s)\}^2 + \{f(0)-f(-s)\}^2  \Big)\,ds \,.
\end{align*}

Our next steps take advantage of the symmetries
\begin{equation}\label{fr}
\beta_{kh}'(s)=\beta_{kh}'(-s)\,,\quad f(s)=f(-s)\,,\quad f(\pi-s)=f(s-\pi)=f(s+\pi)\,,
\end{equation}
granted since both $f$ and $\beta_{kh}'$ are even, $2\pi$-periodic functions.   We also note that
\begin{equation}\label{ib}
\beta_{kh}(\frac{\pi}{2}) > \beta_{kh}(\pi)=0\,,
\end{equation}
because the function $\beta_{kh}$ is odd, $2\pi$-periodic and strictly decreasing on $(0,2\pi)$.

We estimate each of the three terms $I, II$ and $III$ separately.  Note that
 \eqref{f1}-\eqref{f2} yield
\begin{eqnarray*}
&& f^2(\pi)-f(\pi)f(\pi-s)-f^2(0)+f(0)f(s) \\
&& =f(\pi)\{f(\pi)-f(\pi-s)\} + f(0)\{f(s)-f(0)\}\ge 0\quad\text{for}\quad 0 \le s \le \pi\,,
\end{eqnarray*}
Moreover,
for $\frac{\pi}{2} \le s \le \pi$ we have $f(\pi-s)\le f(\frac\pi 2) \le f(s) \le f(\pi)$ so that
\begin{eqnarray*}
&& f^2(\pi)-f(\pi)f(\pi-s)-f^2(0)+f(0)f(s)\\
&& \qquad \ge f^2(\pi)-f(\pi)f\Big(\frac{\pi}{2}\Big)-f^2(0)+f(0)f\Big(\frac{\pi}{2}\Big)  \\
&&\qquad =\{ f(\pi)-f(0)\} \Big\{ f(\pi)+f(0)-f\Big(\frac{\pi}{2}\Big)\Big\} \ge  f(0)\{ f(\pi)-f(0)\} \,.
\end{eqnarray*}
Thus we have the lower bound
\begin{align}
 I &= \int_{-\pi}^\pi \frac{-\beta_{kh}'(s)}{2\pi} \,\{ f^2(\pi)-f(\pi)f(\pi-s)-f^2(0)+f(0)f(-s) \}\,ds  \label{fi2}\\
&= \int_0^\pi \frac{-\beta_{kh}'(s)}{\pi} \,\{ f^2(\pi)-f(\pi)f(\pi-s)-f^2(0)+f(0)f(s) \}\,ds \nonumber \\
&\ge \int_{\pi/2}^\pi \frac{-\beta_{kh}'(s)}{\pi} \,\{ f^2(\pi)-f(\pi)f(\pi-s)-f^2(0)+f(0)f(s) \}\,ds \nonumber \\
&\ge \int_{\pi/2}^\pi \frac{-\beta_{kh}'(s)}{\pi} \, f(0)\{ f(\pi)-f(0)\}\,ds \nonumber
= \frac{\beta_{kh}(\frac{\pi}{2})}{\pi}\, f(0)\{ f(\pi)-f(0)\} \,,\nonumber
\end{align}

Similarly we have
\begin{align}
II &= \int_{-\pi}^\pi \frac{-\beta_{kh}'(s)}{4\pi} \,\{f^2(\pi)-f^2(\pi-s) +f^2(-s)- f^2(0)\}\,ds \label{fi3} \\
&= \int_0^\pi \frac{-\beta_{kh}'(s)}{2\pi} \,\{f^2(\pi)-f^2(\pi-s) +f^2(-s)- f^2(0)\}\,ds\nonumber \\
& \ge \int_{\pi/2} ^\pi \frac{-\beta_{kh}'(s)}{2\pi} \,\{f^2(\pi)-f^2(\pi-s) +f^2(-s)- f^2(0)\}\,ds\nonumber \\
& \ge \int_{\pi/2} ^\pi \frac{-\beta_{kh}'(s)}{2\pi} \,\{f^2(\pi)- f^2(0)\}\,ds =\frac{\beta_{kh}(\frac{\pi}{2})}{2\pi}\,\{f^2(\pi)- f^2(0)\}\,,\nonumber
\end{align}
and, in the case when $\Upsilon>0$,
\begin{align}
 \frac{2}{A\{f(\pi)-f(0)\}}  III &=  \int_{-\pi}^\pi \frac{-\beta_{kh}'(s)}{12\pi} \,\Big(\{f(\pi)-f(\pi-s)\}^2 + \{f(0)-f(s)\}^2 \Big)\,ds \nonumber \\
& = \int_0^\pi \frac{-\beta_{kh}'(s)}{6\pi} \,\Big(\{f(\pi)-f(\pi-s)\}^2 + \{f(0)-f(s)\}^2 \Big)\,ds \nonumber \\
& \ge \int_0^\pi \frac{-\beta_{kh}'(s)}{6\pi} \,\frac{\{f(\pi)-f(\pi-s) +f(s)-f(0)\}^2}{2}\,ds \label{fi4}\\
& \ge \int_{\pi/2}^\pi \frac{-\beta_{kh}'(s)}{6\pi} \,\frac{\{f(\pi)-f(\pi-s) +f(s)-f(0)\}^2}{2}\,ds \nonumber\\
&\ge \int_{\pi/2}^\pi \frac{-\beta_{kh}'(s)}{12\pi} \,\{f(\pi)-f(0)\}^2 \,ds=\frac{\beta_{kh}(\frac{\pi}{2})}{12\pi}\,\{f(\pi)-f(0)\}^2 \,. \nonumber
\end{align}
In the last line we used the fact by \eqref{f1}-\eqref{f2} that
\[
f(s) \ge f(\pi-s) >0 \quad\text{for}\quad \pi \ge s \ge \pi/2\,.
\]
The required result follows from \eqref{123} by combining (\ref{fi2}), (\ref{fi3}) and (\ref{fi4}).
\end{proof}


We are now in a good position to provide the proof of Theorem \ref{mainthm}.

\begin{proof}[Proof of Theorem \ref{mainthm}]
Part (i) was proven in Theorem \ref{gra}.
By \eqref{Fii},
\[ \sup_{\K_-}(\max_\mcs Y -\min_\mcs Y)= \sup_{\K_-} (v(\pi)-v(0)) <\infty.\]
This is Part (ii). Due to \eqref{strpos}, Part (iii) follows at once from \eqref{Fii}.

It remains to prove Part (iv).  We first prove that $m$ and $Q$ are uniformly bounded along $\K_-$.
Indeed, as a consequence of (\ref{fi}), we have that $f$ is uniformly bounded along $\K_-$.
It follows from \eqref{f5} that $Q$ is also bounded, and then, from (\ref{f}) and (\ref{A}),
that $v$ and $a$ are also uniformly bounded along $\K_-$.
Since, as an immediate consequence of (\ref{fc21}), $\J f\geq 0$ everywhere, it follows from (\ref{f44}) that
$B\leq Af-Aa$.  Hence $B$ is bounded above along $\K_-$.

We also observe that, in the notation of Lemma 2, we actually have the equality
\[III=\frac{A}{6} (f(\pi)-f(0))(\J f(\pi) + \J f(0)).\]
When combined with the inequalities $I\geq 0$, $II\geq 0$, this leads to
\[\mcR f\geq \frac{A}{6} (f(\pi)-f(0))(\J f(\pi) + \J f(0)).\]
It is thus a consequence of (\ref{lebo}) that
\[\frac{A}{2}(\J f(\pi) + \J f(0))\leq 3.\]
On the other hand, by the definition of $\mcl f$ and \eqref{fi}, we see that
\[1 + aA +B -\frac{A}{2}(f(\pi)+f(0))+ \frac A2 (\J f(\pi) + \J f(0)) \ge0.\]
Combining the last two inequalities, we have
\[B\geq -4 -aA +\frac{A}{2}(f(\pi)+f(0)) > -4 -aA .\]
Thus $B$ is also bounded away from $-\infty$ all along $\K_-$.

We now prove that the flux $m$ is also bounded along $\K_- $.
In case $\Upsilon >0$, we know from \eqref{B} that
\[m =  \frac{ghB}{\Upsilon} - \Upsilon h^2  + \frac {\Upsilon}2 [v^2]  \]
so that $m$ is bounded along $\K_- $.
In the other case $\Upsilon =0$, we  argue as follows. We write (\ref{add}) in the form
\[\frac{|m|}{kh} = (Q-2gv)^{1/2}\,\left((v')^2
+\Big(\frac{1}{k}+\mcc_{kh}(v')\Big)^2\right)^{1/2}  \le C  \left(|v'|+\Big|\frac{1}{k}+\mcc_{kh}(v')\Big|\right)
\]
on $\R$ for some constant $C$, because $Q$ and $f$ are uniformly bounded.
However by (\ref{nmbv}) and (\ref{graph}) we know that $v'<0$ on $(0.\pi)$
and $\frac{1}{k}+\mcc_{kh}(v')>0$ on $(-\pi,\pi)$.
Thus integrating the inequality on $(-\pi,\pi)$ and using the periodicity of $\mcc_{kh}(v')$,
we obtain the boundedness of $m$.

By Theorem \ref{gra}, alternative $(A_1)$ must necessarily occur in Theorem \ref{glbp}, which means that
 \eqref{coonc} holds. We have shown above that $Q$ and $m$ are bounded along ${\mathcal K}_-$.  Note that, by (\ref{nmbv}) and the evenness of $v$,  we have that
 \[v_s(0)=\max_{x\in\R} v_s(x) \quad\text{for all }s\in (0,\infty)\]
 and therefore we need to prove that
 \be Q-2gv_s(0)\to 0 \quad\text{as }s\uparrow\infty.\label{aws0}\ee
 {\it Arguing by contradiction},  we suppose that there exist $\delta>0$ and a sequence $s_j\uparrow \infty$ such that
 \be Q_j -2gv_{s_j}(x)\geq \delta\quad\text{for all }x\in\R.\label{aws}\ee
 For notational simplicity, we denote in what follows the function $v_{s_j}$ by $v_j$, for any $j\in\N$. We shall prove in what follows that the sequence $\{v_j\}_j$ is bounded in $C^{2,\alpha}_{2\pi}$, a fact which, combined with the previously proved bounds, contradicts the validity of \eqref{coonc}.

 To that aim, we rewrite (\ref{vara}) in the form
 \be
\begin{aligned}\label{aq1}
2(Q-2gv)\mcc_{kh}(v')&=-2g(v\mcc_{kh}(v')-\mcc_{kh}(vv'))\\
&\quad +\Upsilon^2(v^2 \mcc_{kh}(v')+\mcc_{kh}(v^2v')-2v\mcc_{kh}(vv'))\\
 &\quad +\frac{\Upsilon^2}{k}v^2 + 2\Upsilon \,v \,\Big( \frac{m}{kh} - \frac{\Upsilon }{2kh}\,[v^2]\Big)- 2\frac{\Upsilon m}{k} \\
& \quad +\frac{2g}{k}(v-h) - 2g\,[v\,\mcc_{kh}(v')]\\
&=: -2g \J v+ \Upsilon^2 \K v +{\mathcal T}(m,v)
\end{aligned}
\ee
with the operators $\J$ and $\K$ as defined in (\ref{fc21}) and (\ref{fc3}), while ${\mathcal T}(m,v)$ gathers all the remaining terms.
 We also rewrite the inequality (\ref{aza0}), valid along $\K_-$, in the form
 \be\label{aq2}
\Upsilon (v\mcc_{kh}(v')- \mcc_{kh}(vv')) <  -\frac{m}{kh} +\frac{\Upsilon }{2kh}\,[v^2]-\frac{\Upsilon}{k}v.
\ee

We shall make use of the following properties satisfied by the operator $\J$:
\begin{itemize}
\item[(a)] $\J$ maps bounded sets of the periodic Sobolev space $W^{1,2}_{2\pi}$ into bounded sets of $L^\infty_{2\pi}$;
\item[(b)] $\J$ maps bounded sets of $W^{1,p}_{2\pi}$, where $p\in (2,\infty)$, into bounded sets of $C^{1-\frac{2}{p}}_{2\pi}$;
\item[(c)] $\J$ maps bounded sets of $C^{n,\beta}_{2\pi}$ into bounded sets of $C^{
n,\alpha}_{2\pi}$, for any $n\in \N\cup\{0\}$, $\beta\in (0,1)$ and $\alpha\in (0,\beta)$.
\end{itemize}
The proof of these properties follows from the fact that $\mcc_{kh} = \C + \S $,
where $\C$ is the ordinary periodic Hilbert transform and $\S$ is a smoothing operator.
Therefore they are an easy adaptation of the arguments in \cite[Section 10.5]{BT}, where the case of the usual Hilbert transform (formally corresponding to $h=\infty$) is considered, in a similar way to \cite{CV, CSV}.

Observe also, from the explicit representation formulas (\ref{fc21}) and (\ref{fc3}), that we have the inequality
 \be |\K v(x)|\leq \frac{2}{3}(v(0)-v(\pi))|\J v(x)|\quad\text{for all }x\in \R, \label{ekj}\ee
 for any $2\pi$-periodic function that is maximized at $0$ and minimized at $\pi$.

 We prove first that the sequence $\{\J v_j\}_j$ of nonnegative functions is bounded in $L^\infty_{2\pi}$.
 In the case when $\Upsilon>0$, this is an immediate consequence of (\ref{aq2}).
 On the other hand, in the case when $\Upsilon=0$, it follows from (\ref{meana}), using (\ref{aws})
 and the boundedness of $\{m_j\}_j$, that $\{v_j\}_j$ is bounded in
 $W^{1,2}_{2\pi}$, from which the claimed result follows by property (a) above.
 It is now a consequence of (\ref{ekj}) that the sequence $\{\K v_j\}_j$ is also bounded in $L^\infty_{2\pi}$.

In the following we consider the general case $\Upsilon\geq 0$. It is a consequence of (\ref{aq1}),
using (\ref{aws}), that $\{\mcc_{kh} v_j'\}_j$ is bounded in $L^\infty_{2\pi}$, and hence in $L^p_{2\pi}$,
for any $p\in (2,\infty)$. Since the linear operator $\mcc_{kh}$ is bounded and invertible
from $L^p_{2\pi,\circ}$ onto itself (where the subscript $\circ$ is used to denote zero mean),
it follows that the sequence $\{v_j'\}_j$ is bounded in $L^p_{2\pi}$ for any $p\in (2,\infty)$.
Thus $\{v_j\}_j$ is bounded in $W^{1,p}_{2\pi}$ for any $p\in (2,\infty)$.

It is not difficult to prove that $\K$ satisfies the same property expressed by (b) for $\J$.
So $\{\J v_j\}_j$ and $\{\K v_j\}_j$ are bounded in $C^{0,\beta}_{2\pi}$ for any $\beta\in (0,1)$.
When used in (\ref{aq1}) together with (\ref{aws}) and property (b), this
leads to the conclusion that $\{\mcc_{kh} v_j'\}_j$ is bounded in $C^{0,\beta}_{2\pi}$, for any $\beta\in (0,1)$.
Since the linear operator $\mcc_{kh}$ is bounded and invertible from $C^{0,\beta}_{2\pi,\circ}$ onto itself,
it follows that $\{v_j'\}_j$ too is bounded in $C^{0,\beta}_{2\pi}$ for any $\beta\in (0,1)$.

Similarly, $\K$ satisfies the same property expressed by (c) for $\J$. As before, an examination of (\ref{aq1}),
using also (\ref{aws}),
leads to the conclusion that $\{\mcc_{kh} v_j'\}_j$ is bounded in $C^{1,\alpha}_{2\pi}$, for any $\alpha\in (0,1)$.
Since the linear operator $\mcc_{kh}$ is bounded and invertible from $C^{1,\alpha}_{2\pi,\circ}$ onto itself,
it follows that $\{v_j\}_j$ is bounded in $C^{2,\alpha}_{2\pi}$ for any $\alpha\in (0,1)$.
This contradicts (\ref{coonc}), proving therefore that (\ref{aws0}) holds, as required.

\end{proof}

\section{Physical interpretation}

In order to explain the physical relevance of our results, let us note that the prime
sources of currents are winds of long duration (see the discussion in \cite{Jon}).
In deep water these are near-surface shear flows
with about 75 m taken as the reference depth of the layer to which the wind effects are confined,
but in shallower regions a velocity profile can develop throughout the flow.
Approximately 7.5\% of the Earth's ocean floor consists of continental shelves that are practically flat,
with average depth about 60 m and average width of around 65 km.
In some places they are almost nonexistent, while in others,
as along the northern coast of Australia and Siberia, their width exceeds 1000 km; see \cite{Gar}.
A systematic collection of wind drift data has provided guidance for the velocities of wind-generated currents
as functions of the depth \cite{Ew}.  This leads to the current's velocity profile
\begin{equation}\label{wind}
\psi_Y(Y)=\left\{ \begin{array}{l}
 s\,\Big(1+ \displaystyle\frac{Y-h}{Y_0}\Big)\quad\hbox{for}\quad h-Y_0 \le Y \le h\,,\\[0.25 cm]
 0 \quad\hbox{for}\quad 0 \le Y \le h-Y_0\,,
 \end{array}\right.
 \end{equation}
where $Y_0>0$ is the reference depth of the current and where the magnitude $s$ of the surface wind-drift
corresponds to about   2\% of the wind velocity measured at 10 m above sea level.

  In our setting, with $h,\,k >0$ and $\Upsilon  \in {\mathbb R}$ fixed,
  the formulation \eqref{g} of the governing equations
  for waves propagating in the $X$-direction is in a frame of reference moving at the wave speed $c>0$,
  in which the wave is stationary.
  In this frame of reference the underlying current ${\frak U}(Y)$ is defined at every level $Y$ beneath the wave trough
as the average over a wavelength $L$ of the horizontal fluid velocity component,
$${\frak U}(Y)=\frac{1}{L} \int_X^{X+L} \psi_Y(X,Y)\,dX\,,\qquad 0 \le Y \le v(\pi)\,.$$
Using \eqref{g1} and the $L$-periodicity of $\psi$ in the $X$-variable,
we get ${\frak U}_Y=\Upsilon $, so that
\begin{equation}\label{uc}
{\frak U}(Y)={\frak U}_0 + \Upsilon  Y\,,\qquad 0 \le Y \le v(\pi)\,,
 \end{equation}
where ${\frak U}_0$ is the current velocity along the flat bed $Y=0$.
Note the resemblance to the formula \eqref{lb8} for
parallel shear flows. We highlight two types of wind waves propagating in a layer of water
less than 75 m deep, over a flat bed:
\begin{itemize}
\item {\it Downstream waves}
occur when the wind blows in the $X$-direction, generating a current \eqref{wind}-\eqref{uc}
with vorticity $\Upsilon =s/Y_0 >0$ and ${\frak U}_0 \ge 0$.  The resulting flow is called favorable.
\item {\it Upstream waves}
occur when the wind blows in the negative $X$-direction, in which case
the current \eqref{wind}-\eqref{uc} has vorticity $\Upsilon =s/Y_0 <0$ and ${\frak U}_0 \le 0$.
The resulting flow is called adverse.
\end{itemize}

In this context we now discuss the waves of small amplitude that correspond to solutions on the local bifurcation curve,
representing small perturbations of a suitable pure current (a horizontal flow with a flat free surface). The parallel shear flows with
velocity field \eqref{lb8} are labelled by $m$.
For precisely two values of $m$, specified in \eqref{lb7}, they
admit small perturbations in the form of waves with one crest and one trough per wavelength. At the bifurcation point
we have $v \equiv h$, so that \eqref{aza0} simplifies to
\begin{equation}
\pm \Big(\frac{m}{kh} + \frac{\Upsilon  h^2}{2kh} \Big) >0
\end{equation}
		and therefore
the $\pm$ sign in \eqref{aza0} corresponds to the choice of sign in \eqref{lb7}.
Due to \eqref{lamb}-\eqref{lb6},
at the bifurcation point the choice of $\pm$ in \eqref{lb7} amounts to imposing
an inequality $\lb = {\frak U}(h) \gtrless 0$ on the parallel sheer flow.
Now the propagation speed $c$ of wind waves is typically an order of magnitude
greater than the velocity of the surface wind-drift.
Indeed, for wave speeds in the range 10-20 m/s, surface current speeds
of 3 m/s are exceptionally high  \cite{Gar, tk}.
Taking this into account, it means that
$\lb={\frak U}(h) <0$.

  \begin{figure}[h]
  \centering
  \captionsetup{width=.475\linewidth}
  \subfloat[Downstream wave propagation in the physical frame of reference: co-flowing
  (favorable) wave and the underlying
  positively sheared current.]{\includegraphics[width=0.475\textwidth]{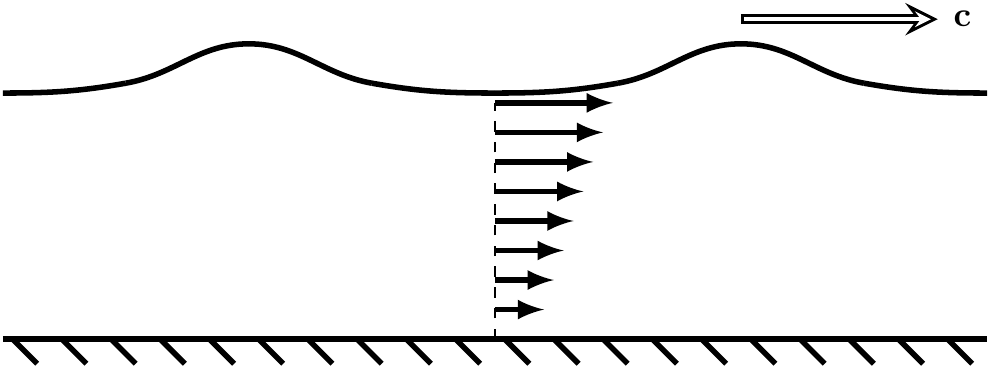}}
  \hfill
  \subfloat[The same flow in the frame of reference moving at the wave speed $c$.]{\includegraphics[width=0.475\textwidth]{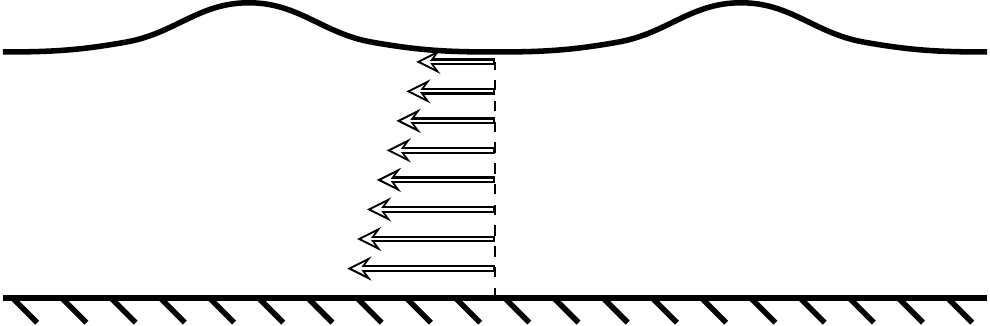}}
  \captionsetup{width=.97\linewidth}
  \caption{\footnotesize Sketches of the surface wave and of the linearly sheared velocity profile of
  the underlying favorable wind-generated current with positive vorticity:
  (A) In the fixed frame of reference the surface wave propagates
  at constant speed $c$, without change of shape, in the $X$-direction. (B) The wave is stationary in a coordinate system moving with the wave speed $c$.}
\end{figure}

Our results in Sections 3 and 4 are valid for favorable flows, which means the choice of the minus sign
in \eqref{aza0} and $\Upsilon  \ge 0$. At the bifurcation point,
$\Upsilon  \ge 0$ then ensures by means
of \eqref{lb8} that $\psi_Y <0$ throughout the pure current flow. Moreover, for the waves
that arise as small perturbations of this pure current state the inequality $\psi_Y <0$ throughout the flow will persist while
for waves of large amplitude along the curve ${\mathcal K}_-$ this inequality is
granted by Theorem \ref{glbp}; see Figure 2 for
the configurations (in the fixed frame and in the frame of reference moving at the wave speed)
of a current generated by wind blowing in the favorable direction,
in water with a flat bed located above the reference depth $Y_0$.

On the other hand, we briefly discuss the interpretation of a flow being adverse, which refers to the alternative choice of
sign in \eqref{aza0}, for solutions on the local bifurcation curve.
For wind-generated adverse currents we argued that ${\frak U}_0 \le 0$ and $\Upsilon  <0$, so that the
waves of small amplitude are perturbations of a pure current of the form \eqref{lb8} with
negative velocity at the surface in the moving frame, given by \eqref{lb6} with the minus sign.
Note that
$m_+^*(\Upsilon) = -m_-^*(-\Upsilon)$.
In a frame of reference at rest, the
current runs in the direction opposite to the waves because ${\frak U}(Y)  < {\frak U}_0 \le 0 $ for $Y>0$; see Figure 3.

 \begin{figure}[h]
  \centering
  \captionsetup{width=.475\linewidth}
  \subfloat[Upstream wave propagation in the physical frame of reference: contra-flowing (adverse) wave
  and the underlying negatively sheared current.]{\includegraphics[width=0.475\textwidth]{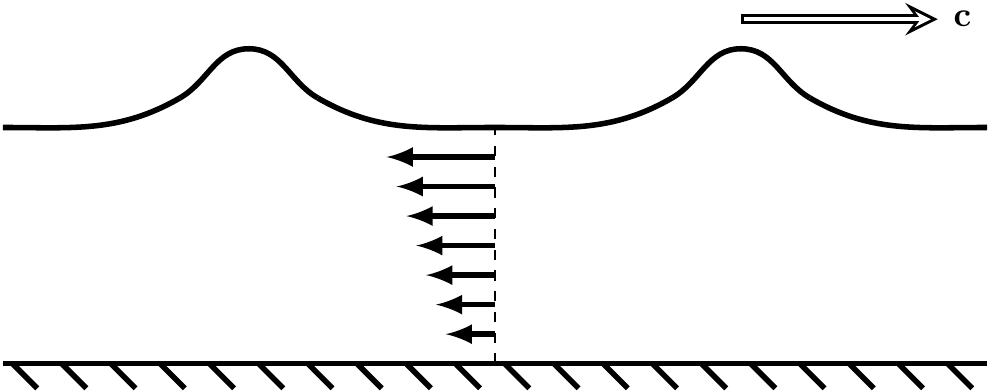}}
  \hfill
  \subfloat[The same flow in the frame of reference moving at the wave speed $c$.]{\includegraphics[width=0.475\textwidth]{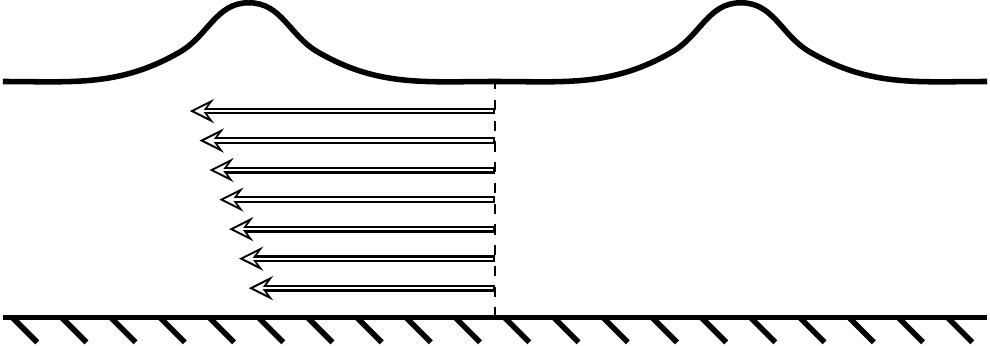}}
  \captionsetup{width=.97\linewidth}
  \caption{\footnotesize Sketches of the surface wave and of the linearly sheared velocity profile of
  the underlying adverse current with negative vorticity: (A) In the fixed frame
  of reference the wave propagates
  at constant speed $c$, without change of shape, in the $X$-direction, counter to the wind-generated current.
  (B) The wave is stationary in the coordinate system moving with the wave speed $c$.}
\end{figure}

Experimental data and numerical simulations \cite{CKS1, CKS2, DP, KS1, KS2, Mo, Rib, Swan, GPT, V} indicate
that the interaction of steady waves with linearly sheared currents can have a significant effect on the waves.
Numerical computations show that a
 favorable shear current typically increases the wavelength and decreases the wave steepness and amplitude,
with a surface wave profile that is markedly asymmetrical about the mean water level,
having long, flat regions near the trough.
On the other hand, an adverse current shortens the wavelength and increases the steepness,
leading to unusual shapes with narrow and peaked crests and overhanging bulbous waves \cite{V}.
In particular, as the wave height increases, vorticity can become a dominant feature in the flow dynamics.
The results presented in this paper can be regarded as
providing a confirmation of these observations for
waves interacting with a favorable current in a flow of constant vorticity over a flat bed.

\section{Appendix}

The purpose of this appendix is to provide some background information
about the reformulations \eqref{sys} and \eqref{add} of the governing equations \eqref{g}.

For any $d>0$, the Hilbert transform operator $\mcc_d$ associated to the strip
\[\mcr_{d}=\{(x,y) \in \R^2:\ -d < y < 0\}\]
is defined for $2\pi$-periodic functions $w \in C^{0,\alpha}_{2\pi,\circ}(\R)$ of zero mean, $[w]=0$, having the Fourier series expansion
\[w(x)=\sum_{n=1}^\infty a_n \cos (nx)\,+\,\sum_{n=1}^\infty b_n \sin (nx),\qquad x \in \R,\]
by
\begin{equation}\label{hfs}
\big(\mcc_d(w)\big)(x)=\sum_{n=1}^\infty a_n \coth
 (nd)\sin(nx)-\sum_{n=1}^\infty b_n\coth (nd)\cos (nx), \qquad x \in \R\,.
\end{equation}
Moreover, for any $w \in C_{2\pi,\circ}^{0,\alpha}(\bdr)$, the representation formula \eqref{fc0} holds (see \cite{CSV}).

The fact that any smooth solution of (\ref{g}) solves \eqref{sys} was established in \cite{CSV} by means of Riemann-Hilbert theory. Moreover, in \cite{CSV} the variational interpretation
of \eqref{sys} was provided: these are precisely the Euler-Lagrange equations of the energy functional
$${\mathcal L}(\Omega,\psi)=\iint_\Omega \Big( \frac{|\nabla\psi|^2}{2} -gY +\Upsilon \psi +\frac{Q}{2}\Big)\,dXdY,$$
associated to the flow. Note that \eqref{g1}-\eqref{g2} and Green's first identity yield
$$\iint_\Omega \Upsilon \psi \,dXdY =\iint_\Omega \psi \Delta \psi \,dXdY = - \iint_\Omega |\nabla\psi|^2\,dXdY - mL\, {\frak U}_0\,,$$
where
$${\frak U}_0=\frac{1}{L} \int_0^L \psi(X,0)\,dX$$
is the velocity of the underlying current along the flat bed $Y=0$ (see the discussion in Section 5).
Consequently we can alternatively express the action
functional ${\mathcal L}(\Omega,\psi)$ as
$${\mathcal L}(\Omega,\psi)=\iint_\Omega \Big(-gY -\frac{|\nabla\psi|^2}{2}  +\frac{Q}{2}\Big)\,dXdY- mL\, {\frak U}_0\,,$$
in which the first two integral terms are the difference between the average potential and kinetic energies
per unit area, while $Q$ is related to the total head.
		More precisely, dividing both
sides of \eqref{g2} by $2g$, all the terms in
$$\frac{\vert\nabla\psi\vert^{2}}{2g}+Y=\frac{Q}{2g} \quad\hbox{on}\quad {\mathcal S}$$
have the dimension of length, the first being called the velocity head and representing the
elevation needed for the fluid to reach the velocity $\vert\nabla\psi\vert$ during frictionless
free fall, and the second term being the elevation head, so that $Q/(2g)$ stands for the total head
(amount of energy per unit weight).
These considerations illustrate the physical relevance of the variational
principle that underlies the formulation \eqref{sys} of the governing equations \eqref{g}
for free-surface flows with constant vorticity.

We now explain how a
smooth solution of (\ref{g}) can be constructed from a solution of (\ref{sys}) if \eqref{pos}-\eqref{m3} hold. First, the fact that any solution of (\ref{sys}) satisfies also (\ref{add}) was established in \cite{CSV} by means of Riemann-Hilbert theory. Now, let us define
\[\mcs=\left\{\left(\frac{x}{k}+\big(\mcc_{kh}(v-h)\big)(x),\, v(x)\right):x\in\bdr\right\},\]
a non-self-intersecting smooth curve contained in the upper half-plane. Moreover, let us consider in the horizontal strip
\[\mcr_{kh}=\{(x,y) \in \R^2:\ -kh < y < 0\}\]
the solution $V$ of
\begin{subequations}\label{VV}
\begin{align}
\Delta V &=0\qquad\hbox{in } \mcr_{kh},\\
V(x,0)&=v(x),\qquad x \in \R,\\
V(x,-kh)&=0,\qquad x \in \R,
\end{align}
\end{subequations}
and let $U$ be a harmonic conjugate of $-V$ in $\mcr_{kh}$, so that $U+iV$ is a holomorphic function there. Then
\begin{equation}\label{v}
\begin{aligned}
U(x+2\pi,y)&=U(x,y)+\displaystyle\frac{2\pi}{k}\qquad\text{for all } (x,y) \in \mcr_{kh},\\
V(x+2\pi,y)&=V(x,y)\qquad\text{for all } (x,y) \in \mcr_{kh}\,,
\end{aligned}
\end{equation}
with the holomorphic function $U+iV$ unique up to an additive real constant. Requiring that $x \mapsto U(x,0)+iV(x,0)$ has
zero mean over one period means that we set
\begin{equation}\label{U}
U(x,0)=\frac{x}{k}+ \big(\mcc_{kh}(v-h)\big)(x),\qquad x \in \bdr.
\end{equation}
It follows that $U+iV$ is a conformal mapping from $\mcr_{kh}$ onto a horizontally periodic domain $\Om$
of period $L=2\pi/k$ and conformal mean depth $h$,
whose upper boundary is $\mcs$ and lower boundary is $\mcb$.
The mapping admits a smooth extension as a homeomorphism
between the closures of these domains, with $\{(x,0):x\in\bdr\}$ being mapped onto $\mcs$
and $\{(x,-kh):x\in\bdr\}$ being mapped onto $\mcb$.
 We recall (see \cite{CV}) that, for a given period $L=2\pi/k$, the conformal mean depth is the unique
positive number $h$ such that there exists a conformal map $U+iV$ from $\mcr_{kh}$ onto $\Omega$, subject
to the periodicity conditions \eqref{v}.

Let us now consider, in the domain $\Omega$ so defined, the unique solution $\psi$ of (\ref{g1})--(\ref{g4}). At the same time, let us introduce the function $\zeta:\mcr_{kh}\to \R$
as the unique periodic solution of the problem
\begin{subequations}\label{gc}
\begin{align}
\Delta\zeta&=0\qquad\hbox{in }\mcr_{kh},\label{gc1}\\
\zeta(x,0)&=m\,-\,\displaystyle\frac{\Upsilon }{2}\,v^2(x)\qquad\hbox{for all } x \in \R,\label{gc3}\\
\zeta(x,-kh)&=0\qquad\hbox{for all } x \in \R.\label{gc4}
\end{align}
\end{subequations}
It is straightforward to check that \eqref{zeta} holds. Moreover, the definitions of $V$ and $\zeta$ and the properties of the Dirichlet--Neumann operator in the strip $\mcr_{kh}$ ensure that, on $y=0$, we have
\be
\begin{aligned}\label{vzc}
V_y &= \displaystyle\frac{1}{k}+\mcc_{kh}(v'),\\
\zeta_y &= \displaystyle\frac{m}{kh}\,-\, \displaystyle\frac{\Upsilon }{2kh}\,[v^2]\,-\,\Upsilon \,\mcc_{kh}(vv').
\end{aligned}
\ee
It is easily seen that (\ref{add}) is merely an equivalent way to express the equality
\begin{equation}
(\zeta_y\,+\,\Upsilon  VV_y)^2=(Q-2gV)(V_x^2+V_y^2)\qquad \hbox{at }  (x,0),\hbox{ for all }  x \in \R,\label{gc2}
\end{equation}
an equality which, on the other hand, is is obtained by simply writing \eqref{g2} in the $(x,y)$-variables. Thus, if $\psi$ is constructed as above from a solution of (\ref{sys}), and hence of (\ref{add}), then is also satisfies (\ref{g2}), as required.

\section{Acknowledgements}
The work of the third author
was supported by a grant of the Romanian Ministery of Research and Innovation, CNCS -
UEFISCDI, project number PN-III-P1-1.1-TE-2016-2314, within PNCDI III.


\end{document}